\documentclass[a4paper,11pt,oneside]{article}
\usepackage{latexsym,amsfonts,amsmath,amssymb,amsthm,verbatim,multicol,endnotes,xcolor}
\usepackage[utf8]{inputenc}
\usepackage{pstricks-add}
\usepackage[normalem]{ulem}
\usepackage{verbatim}

\newcommand{\G}{\Gamma}

\newcommand{\pres}[2]{\langle {#1}\ |\ {#2} \rangle}
\newtheorem{theorem}{Theorem}
\newtheorem{lemma}[theorem]{Lemma}
\newtheorem{prop}[theorem]{Proposition}
\newtheorem{corollary}[theorem]{Corollary}
\newtheorem{maintheorem}{Theorem}

\newtheorem{maincorollary}[maintheorem]{Corollary}
\newtheorem{defn}[theorem]{Definition}
\numberwithin{theorem}{section}
\usepackage{array}
\newcolumntype{L}[1]{>{\raggedright\let\newline\\\arraybackslash\hspace{0pt}}m{#1}}
\newcolumntype{C}[1]{>{\centering\let\newline\\\arraybackslash\hspace{0pt}}m{#1}}
\newcolumntype{R}[1]{>{\raggedleft\let\newline\\\arraybackslash\hspace{0pt}}m{#1}}

\begin{document}
\title{Planar Whitehead graphs with cyclic symmetry arising from the study of Dunwoody manifolds}%
\author{James Howie\thanks{Howie was supported for part of this project by Leverhulme Trust Emeritus Fellowship EM-2018-023$\backslash$9.}~~and Gerald Williams\thanks{Williams was supported for part of this project by Leverhulme Trust Research Project Grant RPG-2017-334.}}

\maketitle

\begin{abstract}
A fundamental theorem in the study of Dunwoody manifolds is a classification of finite graphs on $2n$ vertices that satisfy seven conditions (concerning planarity, regularity, and a cyclic automorphism of order~$n$). Its significance is that if the presentation complex of a cyclic presentation is a spine of a 3-manifold then its Whitehead graph satisfies the first five conditions (the remaining conditions do not necessarily hold). In this paper we observe that this classification relies implicitly on an unstated 8th condition and that this condition is not necessary for such a presentation complex to be the spine of a 3-manifold. We expand the scope of Dunwoody's classification by classifying all graphs that satisfy the first five conditions.
\end{abstract}



\section{Introduction}\label{sec:newintro}

A fundamental theorem in the study of so-called \em Dunwoody manifolds, \em which leads to the definition of a \em Dunwoody diagram \em (in~\cite{CattabrigaMulazzaniVesnin}), is Theorem~1 of~\cite{Dunwoody}. That theorem asserts a classification of all finite graphs $\Gamma$ that satisfy the following seven conditions:

\begin{itemize}
  \item[($1$)] $\Gamma$ is planar;
  \item[($2$)] $\Gamma$ is regular;
  \item[($3$)]
  \begin{itemize}
    \item[(a)] $\Gamma$ has no loops (i.e. no edge from a vertex to itself),
    \item[(b)] there may be more than one edge joining the same pair of vertices;
  \end{itemize}
  \item[($4$)] $\Gamma$ has $2n$ vertices, for some positive integer $n$;
  \item[($5$)] $\Gamma$ admits an automorphism $\theta$ acting as a regular permutation of order $n$ both on the edges and the vertices of $\Gamma$;
  \item[($6$)] in some $\theta$-orbit there is a pair of vertices that are joined by an edge;
  \item[($7$)] there are two vertices in different $\theta$-orbits that are joined by an edge.
\end{itemize}

(A permutation is \em regular \em if it is either trivial or has no fixed points and is the product of disjoint cycles of the same length.) Unfortunately, for reasons discussed below, the theorem as originally stated omits an implicitly assumed, but unstated, 8th condition and is false without it (a counterexample is given in Figure~\ref{fig:counterexample} where $\theta=(1,2,3,4,5,6)(7,8,9,10,11,12)$):

\begin{itemize}
  \item[($8$)] If $v,v'$ are adjacent vertices in the two orbits of $\theta$ then there is an edge joining $v$ and $\theta^k(v)$ for some $(k,n)=1$ and there is an edge joining $v'$ and $\theta^m(v')$ for some $(m,n)=1$.
\end{itemize}

As we now describe, the motivation for the introduction of these conditions in~\cite{Dunwoody} arose from the study of cyclic presentations of groups that correspond to spines of 3-manifolds, or equivalently, Heegaard splittings of 3-manifolds with cyclic symmetry.  The fact that condition~($8$) was overlooked in~\cite{Dunwoody} appears to be a consequence of the particular type of cyclic symmetry considered there.

The \em presentation complex \em $K$ of a group presentation $P=\pres{X}{R}$ is the 2-complex with one 0-cell $O$, a loop at $O$ for each generator $x\in X$ and a 2-cell for each relator (the boundary of that 2-cell spelling the relator). If $R$ is a regular neighbourhood of $O$ then $K\cap \partial R$ is a 1-dimensional cell complex called the \em Whitehead graph \em or \em link graph \em of $P$. Thus the Whitehead graph of ${P}$ is the graph with $2|X|$ vertices $v_x,v_x'$ ($x \in X$) and an edge $(v_x,v_y)$ (resp. $(v_x',v_y')$, $(v_x,v_y')$) for each occurrence of a cyclic subword $xy^{-1}$ (resp.\,$x^{-1}y$, $(xy)^{\pm 1}$) in a relator $r\in R$.  In~\cite{Neuwirth} Neuwirth provides an algorithm for deciding if the presentation complex of a group presentation with an equal number of generators and relators is a spine of a closed compact 3-manifold~$M$. A necessary condition for a presentation to correspond to a spine in this way is that its Whitehead graph is planar (see, for example, \cite[page~33]{AngeloiMetzler}). Since multiedges of a Whitehead graph clearly have no effect on the graph's planarity, it is convenient to define the \em reduced Whitehead graph \em of ${P}$ to be the graph obtained from the Whitehead graph of ${P}$ by replacing all multiedges between two vertices of the Whitehead graph of ${P}$ by a single edge.

A  cyclic presentation of a group is a presentation with an equal number of generators and relators that admits a cyclic symmetry. Presentations with an equal number of generators and relators are called \em balanced presentations \em and are of interest from both topological and algebraic perspectives. Cyclic presentations form an important subclass that are more tractable than arbitrary balanced presentations (the built-in symmetry can be exploited) and which can correspond to spines of manifolds that inherit the cyclic symmetry.

Formally, the \em cyclic presentation \em $P_n(w)$ is the presentation with generators $x_0,\ldots ,x_{n-1}$ (say) and relators $w(x_i,x_{i+1},\ldots, x_{i+n-1})$ ($0\leq i\leq n-1$, subscripts mod~$n$), where the \em defining word \em $w(x_0,x_1,\ldots ,x_{n-1})$ is some fixed word $w$ in the generators. It follows that the Whitehead graph $\Gamma$ of $P_n(w)$ satisfies (2),(3(b)),(4),(5), and if $w$ is cyclically reduced (i.e. if it has no cyclic subwords of the form $xx^{-1}$ or $x^{-1}x$ for any generator $x$) then $\Gamma$ satisfies (3(a)) and if it corresponds to a spine of a manifold then (1) also holds. The additional conditions ($6$),($7$) and ($8$) were imposed in order to obtain Dunwoody's classification theorem (Theorem~\ref{thm:DunwoodyOriginal}, below). The implicitly assumed condition ($8$) is not necessary for a cyclic presentation to correspond to a spine. We know of two (reduced) Whitehead graphs of cyclic presentations corresponding to spines of manifolds that demonstrate this. The cyclic presentation $P_{2m}(x_0x_1x_2^{-1})$ of the Fibonacci group $F(2,2m)$ corresponds to a spine of the \em Fibonacci manifold \em (see~\cite{CavicchioliSpaggiari},\cite{HKM}) and its Whitehead graph is of our type~(II.11), given in Figure~\ref{fig:(II.11)}. The cyclic presentation $P_{2m}(x_1(x_2^{-1}x_0)^l)$ ($l\geq 1$) corresponds to a spine of a manifold~\cite{Jeong},\cite{JeongWang} and its reduced Whitehead graph is of type~(II.14), given in Figure~\ref{fig:(II.14)}. As noted in~\cite{Dunwoody}, the Fibonacci manifolds satisfy ``a different sort of cyclic symmetry'' to the Dunwoody manifolds. The principal difference in symmetry seems to be that the generating symmetry of the Fibonacci manifolds is orientation-reversing rather than orientation-preserving. Indeed, the classification in~\cite{Dunwoody} is correct under the slightly stronger hypothesis that the symmetry in~($5)$ extends to an orientation-preserving self-homeomorphism of the ambient 2-sphere into which $\Gamma$ embeds. A spine $C$ may correspond to a cyclic presentation where the defining word $w$ is not cyclically reduced, in which case the Whitehead graph contains loops (in contrast to condition ($3$(a))). For example, the cyclic presentation $\pres{x_0}{x_0^2x_0^{-1}}$ corresponds to a spine of the 3-sphere $S^3$~\cite{Zeeman}, and the (Fibonacci) cyclic presentation $F(2,2)=\pres{x_0,x_1}{x_0x_1x_0^{-1}, x_1x_0x_1^{-1}}$ corresponds to a spine of $S^3$ by~\cite[Theorem~3]{CavicchioliSpaggiari},\cite[Theorem~A]{HKM}.

In this paper we correct and expand~\cite[Theorem~1]{Dunwoody}, by classifying all finite graphs $\Gamma$ that satisfy conditions ($1$)--($5$), but not necessarily conditions ($6$),($7$),($8$). The impact of the original omission of condition~(8) on the literature concerning Dunwoody manifolds is likely to be minor, as the significance of~\cite{Dunwoody} is the generation of interesting examples of cyclic presentations that correspond to spines of manifolds, rather than the classification. However, our classification expands the scope of Dunwoody's work by providing many new graphs that potentially arise as Whitehead graphs of such presentations. It is currently unknown which of these graphs do indeed arise in this way, beyond a few known cases which we describe at the end of this section.

To avoid a detailed tracking of the number of edges between every pair of vertices, instead of considering graphs that satisfy $(3)$ we consider graphs that satisfy the following variant:
\begin{itemize}
  \item[($3'$)]
  \begin{itemize}
    \item[(a)] $\Gamma$ has no loops (i.e. no edge from a vertex to itself),
    \item[(b)] there is at most one edge joining the same pair of vertices.
  \end{itemize}
\end{itemize}
The process of replacing all edges between any pair of vertices of a regular graph by a single edge may result in a graph that is not regular so to adjust for our variant condition ($3'$) we modify condition~($2$) to the following:
\begin{itemize}
  \item[($2'$)] $\Gamma$ can be made to be regular by replacing each edge of $\Gamma$ joining two of its vertices by a suitable (finite, positive) number of edges between those vertices in such a way that, given any pair of vertices $u,v$, the number of edges joining $u,v$ is equal to the number of vertices joining $\theta(u),\theta(v)$.
\end{itemize}

A classification of all finite graphs that satisfy conditions ($1$),($2$),($3$),($4$),($5$) can easily be obtained from one of all finite graphs that satisfy ($1$),($2'$),($3'$),($4$),($5$). The purpose of this paper is to obtain a classification of graphs satisfying the latter conditions.

Note that condition~($6$) appears somewhat asymmetric, in that it involves only one of the two orbits of vertices under the automorphism $\theta$. However, under the presence of condition ($2'$), if ($6$) holds then, in fact, for each $\theta$-orbit of vertices, there exists a pair of edges in that orbit that are joined by an edge.

A classification of all finite graphs that satisfy conditions ($1$),($2$),($3$),($4$),($5$),($6$),($7$),($8$)
is equivalent to one of all graphs that satisfy conditions ($1$),($2'$),($3'$),($4$),($5$),($6$),($7$),($8$). In this re-expression, and referring to~\cite[Figure~1]{Dunwoody} where a graph $\Gamma_n(a,b,c)$ is defined, Dunwoody's classification~\cite[Theorem~1]{Dunwoody}, can be stated as follows:

\begin{theorem}[{\cite[Theorem~1]{Dunwoody}}]\label{thm:DunwoodyOriginal}
Let $\Gamma$ be a graph satisfying ($1$),($2'$),($3'$),($4$),($5$),($6$),($7$),($8$). Then $\Gamma$ is isomorphic to either $\Gamma_n(1,1,1)$ or $\Gamma_n(1,1,0)$.
\end{theorem}

(The graphs $\Gamma_n(1,1,1)$ and $\Gamma_n(1,1,0)$ correspond to the graphs in Figures~\ref{fig:(I.1)} and~\ref{fig:(I.3)} (below), respectively.)
We express our results in terms of graphs $\Gamma_n(\mathcal{A},\mathcal{B},\mathcal{Q})$ that we now define.

\begin{defn}\label{def:Gamma_n(A,B,Q)}
\em Let $n\geq 1$, $\mathcal{A},\mathcal{B},\mathcal{Q}\subseteq\{0,\ldots ,n-1\}$ and define the graph $\Gamma_n(\mathcal{A},\mathcal{B},\mathcal{Q})$ to be the graph with vertices $v_0,\ldots ,v_{n-1}, v_0',\ldots ,v_{n-1}'$ and edges $(v_{i},v_{i+a})$, $(v_{i}',v_{i+b}')$, $(v_{i},v_{i+q}')$,  for all $0\leq i\leq n-1$, $a\in \mathcal{A},b\in \mathcal{B},q\in \mathcal{Q}$ (subscripts mod~$n$). We shall say that $\Gamma_n(\mathcal{A},\mathcal{B},\mathcal{Q})$ is \em properly given \em if  $a_1\not \equiv \pm a_2$~mod~$n$ and $b_1\not \equiv \pm b_2$~mod~$n$ for all $a_1,a_2\in \mathcal{A}$ and all $b_1,b_2\in \mathcal{B}$.\em
\end{defn}

The term `properly given' was introduced by Heuberger~\cite{Heuberger} when studying circulant matrices, and is a useful concept since $(v_i,v_{i+a})$ (resp. $(v_i',v_{i+b}')$) is an edge of $\Gamma_n(\mathcal{A},\mathcal{B},\mathcal{Q})$ if and only if $(v_i,v_{i-a})$ (resp. $(v_i',v_{i-b}')$) is an edge, and so there is no redundancy in the sets $\mathcal{A},\mathcal{B}$ when $\Gamma_n(\mathcal{A},\mathcal{B},\mathcal{Q})$ is properly given. In Lemma~\ref{lem:3457isGamma} we show that if a graph satisfies ($3'$),($4$),($5$) then it is isomorphic to $\Gamma_n(\mathcal{A},\mathcal{B},\mathcal{Q})$ for some $\mathcal{A},\mathcal{B}\subseteq\{1,\ldots ,n-1\}$ and $\mathcal{Q}\subseteq\{0,\ldots ,n-1\}$; if it also satisfies ($2'$) then  it is the Whitehead graph of some cyclic presentation $P_n(w)$ and either $\mathcal{A}=\mathcal{B}=\emptyset$ or $\mathcal{A}\neq \emptyset$ and $\mathcal{B}\neq \emptyset$, and if it also satisfies ($7$) then $\mathcal{Q}\neq \emptyset$. If ($7$) does not hold, then $\Gamma_n(\mathcal{A},\mathcal{B},\mathcal{Q})$ is the disjoint union of two circulant graphs, in which case a classification may be obtained directly from a result of Heuberger~\cite{Heuberger} (Lemma~\ref{thm:planarcirculant}, below); thus we may assume that ($7$) does hold. We show in Lemma~\ref{lem:connectedGamma} that all components of $\Gamma_n(\mathcal{A},\mathcal{B},\mathcal{Q})$ are then isomorphic, so we may also assume that $\Gamma_n(\mathcal{A},\mathcal{B},\mathcal{Q})$ is connected. The regularity condition ($2'$) is unimportant for our proof methods so we provide a more general classification that is independent of ($2'$) before specializing to the more restricted case. Specifically, Theorem~\ref{mainthm:planarWhiteheadgraphs} classifies all the connected graphs that satisfy ($1$),($3'$),($4$),($5$) and Corollary~\ref{maincor:allconditions} classify those that also satisfy ($2'$). Dunwoody's classification (Theorem~\ref{thm:DunwoodyOriginal}) follows as a special case of Corollary~\ref{maincor:allconditions}.

\begin{maintheorem}\label{mainthm:planarWhiteheadgraphs}
  Suppose $n\geq 2$, $\mathcal{A},\mathcal{B}\subseteq\{1,\ldots ,n-1\}$, $\emptyset\neq \mathcal{Q}\subseteq\{0,\ldots ,n-1\}$ and let $\Gamma=\Gamma_n(\mathcal{A},\mathcal{B},\mathcal{Q})$ be properly given. Then $\Gamma$ is connected and planar if and only if one of the cases of Table~\ref{table:planarWHG} holds for some $0\leq q\leq n-1$.
\end{maintheorem}

\begin{maincorollary}\label{maincor:allconditions}
  Suppose $n\geq 2$, $\mathcal{A},\mathcal{B}\subseteq\{1,\ldots ,n-1\}$, $\emptyset\neq \mathcal{Q}\subseteq\{0,\ldots ,n-1\}$ and let $\Gamma=\Gamma_n(\mathcal{A},\mathcal{B},\mathcal{Q})$ be properly given. Then $\Gamma$ is connected and satisfies ($1$) and ($2'$) if and only if one of the cases of Table~\ref{table:planarWHG} holds for some $0\leq q\leq n-1$ where either $\mathcal{A}=\mathcal{B}=\emptyset$ or $\mathcal{A}\neq \emptyset$ and $\mathcal{B}\neq \emptyset$.
\end{maincorollary}

We exhibit the graphs $\Gamma_n(\mathcal{A},\mathcal{B},\mathcal{Q})$ from Theorem~\ref{mainthm:planarWhiteheadgraphs} and Corollary~\ref{maincor:allconditions} in Figures~\ref{fig:(I.1)}--\ref{fig:(III.14)} where we denote the vertices $v_i$ and $v_i'$ by $i$ and $\bar{i}$ (for clarity, in Figures~\ref{fig:(III.1)} and~\ref{fig:(III.2)} we omit some vertex labels that can be easily deduced). These fall into 3 classes: Class~I, which holds for all $n$, Class~II which holds for all even $n$, and Class~III which holds for all $n\equiv 2$~mod~$4$. We now summarize the situation for those graphs from Corollary~\ref{maincor:allconditions}, i.e.\,those that are potentially reduced Whitehead graphs of cyclic presentations corresponding to spines of manifolds.
In Class~I there are three isomorphism classes of connected graphs that satisfy ($1$),($2'$),($3'$),($4$),($5$),($7$). Two of these (types (I.1) and (I.3)) satisfy both ($6$) and ($8$) and correspond to Dunwoody manifolds, and one of them (type~(I.5)) does not satisfy ($6$)  (and hence does not satisfy ($8$)). It follows that Dunwoody's original classification~\cite{Dunwoody} is correct when $n$ is odd. In Class~II there are 7 isomorphism classes of connected graphs that satisfy ($1$),($2'$),($3'$),($4$),($5$),($7$); none of these satisfy (8), all but one of them (type~(II.6)) satisfies ($6$). We remark that, when $n$ is even, conditions (I.3) and (II.6) yield isomorphic graphs: if $n$ is even and $(n,s)=1$ then $\Gamma_n(\{s\},\{s\},\{q\})\cong \Gamma_n(\emptyset,\emptyset,\{q,q+s,q-s\})$. As previously mentioned, two of the graphs in Class~II are known to correspond to spines of manifolds: type~(II.11) to the Fibonacci manifolds of~\cite{HKM} and type~(II.14) to the manifolds of~\cite{Jeong},\cite{JeongWang}. In Class~III there are 9 isomorphism classes of connected graphs that satisfy ($1$),($2'$),($3'$),($4$),($5$),($7$); precisely two of them satisfy (8), all of them satisfy ($6$) and none are known to correspond to spines of manifolds. We expect that some of the remaining graphs are also reduced Whitehead graphs of cyclic presentations corresponding to spines of manifolds; investigating which of them are is a subject for future research.

\section{Circulant graphs and the graphs $\Gamma_n(\mathcal{A},\mathcal{B},\mathcal{Q})$}\label{sec:circulants}

We first prove the previously mentioned lemma, which relates the conditions ($2'$),($3'$),($4$),($5$),($6$),($7$) to the graphs $\Gamma_n(\mathcal{A},\mathcal{B},\mathcal{Q})$ and to Whitehead graphs of cyclic presentations.

\begin{lemma}\label{lem:3457isGamma}
Let $\Gamma$ be a graph.
\begin{itemize}
  \item[(a)] The following are equivalent:
\begin{itemize}
  \item[(i)] conditions ($3'$(b)),($4$),($5$) (resp.\,and ($3'$(a))) hold;
  \item[(ii)] $\Gamma$ is isomorphic to $\Gamma_n(\mathcal{A},\mathcal{B},\mathcal{Q})$ for some $\mathcal{A},\mathcal{B},\mathcal{Q}\subseteq \{0,\ldots, n-1\}$ (resp.\,where $0\not \in \mathcal{A}$ and $0\not \in \mathcal{B}$).
\end{itemize}
\item[(b)] The following are equivalent:
\begin{itemize}
  \item[(i)] conditions ($2'$),($3'$(b)),($4$),($5$) (resp.\,and ($3'$(a))) hold;
  \item[(ii)] $\Gamma$ is isomorphic to $\Gamma_n(\mathcal{A},\mathcal{B},\mathcal{Q})$ for some $\mathcal{A},\mathcal{B},\mathcal{Q}\subseteq \{0,\ldots, n-1\}$ where either $\mathcal{A}=\emptyset$ and $\mathcal{B}=\emptyset$ or $\mathcal{A}\neq \emptyset$ and $\mathcal{B}\neq \emptyset$  (resp.\,where $0\not \in \mathcal{A}$ and $0\not \in \mathcal{B}$);
  \item[(iii)] $\Gamma$ is isomorphic to the reduced Whitehead graph of a cyclic presentation $P_n(w)$ (resp.\,where $w$ is cyclically reduced).
\end{itemize}
\end{itemize}
\end{lemma}

\begin{proof}
In each case the `resp.' statements are obvious, so we pay no attention to them in what follows.

(a)(i) $\Rightarrow$ (a)(ii). Let $v,v'$ be vertices of $\Gamma$ from the two orbits of $\theta$. Since the automorphism $\theta$ of condition (5) acts as a regular permutation, for distinct $i,j$ ($0\leq i,j\leq n-1$) we have $\theta^i(v)\neq \theta^j(v)$ and $\theta^i(v')\neq \theta^j(v')$. Thus, for each $0\leq i\leq n-1$ setting $v_i=\theta^i(v)$, $v_i'=\theta^i(v')$ defines $2n$ distinct vertices. The set of neighbours of vertex $v$ is the set $N(v)=\{v_a\ |\ a\in \mathcal{A}\} \cup \{v_q'\ |\ q\in \mathcal{Q}\}$ for some $\mathcal{A}, \mathcal{Q}\subseteq \{0,\ldots ,n-1\}$. Similarly the set of neighbours of vertex $v'$ is the set $N(v')=\{v_b'\ |\ b\in \mathcal{B}\} \cup \{v_{n-q}\ |\ q\in \mathcal{Q}\}$ for some $\mathcal{B}\subseteq \{0,\ldots ,n-1\}$. Thus, by ($5$), for each $0\leq i\leq n-1$ the set of neighbours $N(v_i)=\{v_{i+a}\ |\ a\in \mathcal{A}\} \cup \{v_{i+q}'\ |\ q\in \mathcal{Q}\}$ and $N(v_i')=\{v_{i+b}'\ |\ b\in \mathcal{B}\} \cup \{v_{i-q}\ |\ q\in \mathcal{Q}\}$ (subscripts mod~$n$). That is, $\Gamma$ is isomorphic to $\Gamma_n(\mathcal{A},\mathcal{B},\mathcal{Q})$.

(a)(ii) $\Rightarrow$ (a)(i). Letting $\theta$ be the automorphism given by $\theta(v_i)=v_{i+1}$, $\theta(v'_i)=v'_{i+1}$ (subscripts mod~$n$) we see that ($5$) holds while $(3'(b))$, $(4)$ clearly hold.

(b)(i) $\Rightarrow$ (b)(ii). By part~(a) the graph $\Gamma$ is isomorphic to $\Gamma_n(\mathcal{A},\mathcal{B},\mathcal{Q})$ for some $\mathcal{A},\mathcal{B},\mathcal{Q}\subseteq \{0,\ldots, n-1\}$. If $\mathcal{A}=\emptyset$ and $\mathcal{B} \neq \emptyset$ then any attempt to make $\Gamma=\Gamma_n(\mathcal{A},\mathcal{B},\mathcal{Q})$ regular by the process in condition ($2'$) will fail: vertices $v_i$ are only adjacent to vertices $v_j'$ and each edge at a vertex $v_i$ is incident at a vertex $v_j'$, while there exist edges $(v_j',v_{j+b}')$, ensuring that the degree of $v_j'$ will strictly exceed that of $v_i$. A similar argument applies if $\mathcal{A}\neq \emptyset$ and $\mathcal{B}= \emptyset$.

(b)(ii) $\Rightarrow$ (b)(i). By part~(a) the conditions ($3'$(b)),($4$),($5$) hold, so it remains to show that ($2'$) holds. If $|\mathcal{A}|=|\mathcal{B}|$ then $\Gamma_n(\mathcal{A},\mathcal{B},\mathcal{Q})$ is regular so ($2'$) holds. If  $|\mathcal{A}|>|\mathcal{B}|>0$ then, for some fixed $b\in\mathcal{B}$ replacing the edge joining vertices $v'_i,v'_{i+b}$ by $|\mathcal{A}|-|\mathcal{B}|+1$ edges joining these vertices (for each $0\leq i\leq n-1$) produces a regular graph, obtained by the manner described in ($2'$). A similar argument applies if $|\mathcal{B}|>|\mathcal{A}|>0$.

(b)(i) $\Rightarrow$ (b)(iii) By part~(a) the graph $\Gamma$ is isomorphic to $\Gamma_n(\mathcal{A},\mathcal{B},\mathcal{Q})$ for some $\mathcal{A},\mathcal{B},\mathcal{Q}\subseteq \{0,\ldots, n-1\}$. Consider a regular graph $\Lambda$ obtained from $\Gamma=\Gamma_n(\mathcal{A},\mathcal{B},\mathcal{Q})$ in the manner described in ($2'$). Then $\Lambda$ has vertices $v_i,v'_i$ and edges $(v_i,v_{i+a})$ of multiplicity $m_a\geq 1$, $(v'_i,v'_{i+b})$ of multiplicity $m_b\geq 1$, and $(v_i,v'_{i+q})$ of multiplicity $m_q\geq 1$ (say) for each $a\in\mathcal{A}$, $b\in\mathcal{B}$, $q\in\mathcal{Q}$ ($0\leq i\leq n-1$). Then the degree of any vertex $v_i$ is $2\sum_{a\in\mathcal{A}} m_a +\sum_{q\in\mathcal{Q}} m_q$ and the degree of any vertex $v'_i$ is $2\sum_{b\in\mathcal{B}} m_b +\sum_{q\in\mathcal{Q}} m_q$ and so, since $\Lambda$ is regular, we have $\sum_{a\in\mathcal{A}} m_a = \sum_{b\in\mathcal{B}} m_b$. By replacing each edge of $\Lambda$ by $n$ edges, if necessary, we may assume that $\sum_{a\in\mathcal{A}} am_a +\sum_{b\in\mathcal{B}} bm_b + \sum_{q\in\mathcal{Q}} qm_q\equiv 0$~mod~$n$. This ensures that there exists a word $w$ in generators $x_0,\ldots ,x_{n-1}$ with $m_a$ cyclic subwords of the form $x_jx_{j+a}^{-1}$, $m_b$ cyclic subwords of the form $x_j^{-1}x_{j+b}$, and $m_q$ cyclic subwords of the form $x_jx_{j+q}$ for each $a\in\mathcal{A}$, $b\in\mathcal{B}$, $q\in\mathcal{Q}$. Then setting $v_i=v_{x_i},v'_i=v'_{x_i}$ we see that $\Lambda$ is isomorphic to the Whitehead graph of the cyclic presentation $P_n(w)$.

(b)(iii) $\Rightarrow$ (b)(i). Suppose that $\Gamma$ is the reduced Whitehead graph of a cyclic presentation $P_n(w)$. Then, as described in the introduction, the Whitehead graph $\Lambda$ of $P_n(w)$ satisfies ($2$),($3$(b)),($4$),($5$). Thus $\Gamma$ satisfies ($3'$(b)),($4$),($5$), and since the regular graph $\Lambda$ can be obtained from $\Gamma$ in the manner described in~($2'$) the graph $\Gamma$ satisfies ($2'$)
\end{proof}

The \em circulant graph \em $\mathrm{circ}_n(\mathcal{S})$ (where $\mathcal{S}\subseteq \{ 0,1,\dots , n-1\}$) is the graph with vertex set $\{ u_0,\ldots ,u_{n-1} \}$ and edge set $\{ (u_i, u_{i+s})\ |\ 0\leq i\leq n-1, s\in \mathcal{S}\}$ (subscripts mod~$n$). Since $(u_i,u_{i+s})$ is an edge of $\mathrm{circ}_n(S)$ if and only if $(u_i,u_{i-s})$ is an edge, the graph $\mathrm{circ}_n(\mathcal{S})$ is said to be \em properly given \em if $s\not \equiv \pm t$~mod~$n$ for all $s,t\in \mathcal{S}$, $s\neq t$~\cite{Heuberger}. Note that the induced subgraphs of $\Gamma_n(\mathcal{A},\mathcal{B},\mathcal{Q})$ with vertex set $\{v_0,\ldots ,v_{n-1}\}$ and with vertex set $\{v_0',\ldots ,v_{n-1}'\}$ are the circulant graphs $\mathrm{circ}_n(\mathcal{A})$, $\mathrm{circ}_n(\mathcal{B})$, respectively. In particular, if $\mathcal{Q}=\emptyset$ then $\Gamma_n(\mathcal{A},\mathcal{B},\mathcal{Q})$ is the disjoint union of $\mathrm{circ}_n(\mathcal{A})$ and $\mathrm{circ}_n(\mathcal{B})$. The connected components of a circulant graph are described by the following result.

\begin{lemma}[\cite{BoeschTindell}]\label{lem:connectedcirc}
Let $\mathcal{S}\subseteq \{0,1,\ldots ,n-1\}$ and set $d=\mathrm{gcd}(n, s\ (s\in S))$. Then $\mathrm{circ}_n(\mathcal{S})$ consists of $d$ connected components, each of which is isomorphic to $\mathrm{circ}_{n/d}(\mathcal{S}/d)$ where $\mathcal{S}/d=\{ s/d\ |\ s\in \mathcal{S}\}$. In particular, $\mathrm{circ}_n(\mathcal{S})$ is connected if and only if $d=1$.
\end{lemma}

In a similar way, for $\mathcal{Q}\neq \emptyset$, the connected components of $\Gamma_n(\mathcal{A},\mathcal{B},\mathcal{Q})$ are described by the following. (When $\mathcal{Q}=\emptyset$, the connected components of $\Gamma_n(\mathcal{A},\mathcal{B},\mathcal{Q})$ follow immediately from Lemma~\ref{lem:connectedcirc}.)

\begin{lemma}\label{lem:connectedGamma}
Suppose $\mathcal{Q}\neq\emptyset$ and fix some $q_0\in\mathcal{Q}$ and set $d=\mathrm{gcd}(n,a\ (a\in \mathcal{A}), b\ (b\in \mathcal{B}), q-q_0\  (q\in \mathcal{Q}))$. Then $\Gamma=\Gamma_n(\mathcal{A},\mathcal{B},\mathcal{Q})$ consists of $d$ connected components, each of which is isomorphic to $\Gamma_{n/d}(\mathcal{A}/d,\mathcal{B}/d,\mathcal{Q}/d)$ where $\mathcal{A}/d=\{a/d\ |\ a\in \mathcal{A}\}$, $\mathcal{B}/d=\{b/d\ |\ b\in \mathcal{B}\}$, $\mathcal{Q}/d=\{q_0+(q-q_0)/d\ |\ q\in \mathcal{Q}\}$. In particular, $\Gamma$ is connected if and only if $d=1$.
\end{lemma}

\begin{proof}
Let $\Lambda_0=\Gamma_{n/d}(\mathcal{A}/d,\mathcal{B}/d,\mathcal{Q}/d)$. Then $\Lambda_0$ has vertices $x_j,x_j'$ and edges  $(x_j,x_{j+a/d})$, $(x_j',x_{j+b/d}')$, $(x_j,x_{j+q_0+(q-q_0)/d}')$ $a\in \mathcal{A}, b\in \mathcal{B}, q\in \mathcal{Q}$, $0\leq j\leq n/d-1$. For each $0\leq j\leq n/d-1$ let $y_j=x_j$, $y_j'=x_{j+q_0}'$. Then $\Lambda_0$ has vertices $y_j,y_j'$ and edges  $(y_j,y_{j+a/d})$, $(y_j',y_{j+b/d}')$, $(y_j,y_{j+(q-q_0)/d}')$ $a\in \mathcal{A}, b\in \mathcal{B}, q\in \mathcal{Q}$, $0\leq j\leq n/d-1$. Contracting each edge $(y_j,y_j')$ contracts $\Lambda_0$ to the circulant graph $\mathrm{circ}_n(\mathcal{A}/d\cup \mathcal{B}/d \cup \{ (q-q_0)/d\ (q\in \mathcal{Q}) \})$, which is connected by Lemma~\ref{lem:connectedcirc}, and thus $\Lambda_0$ is connected.

For each $j$ set $w_{dj}=y_j$, $w_{dj}'=y_j'$, so that $\Lambda_0$ has vertices $w_0,w_d,\ldots ,w_{n-d}$, $w_0',w_d',\ldots ,w_{n-d}'$ and edges $(w_i,w_{i+a})$, $(w_i',w_{i+b}')$, $(w_i,w_{i+q-q_0}')$, where $i=dj$. Now for each $1\leq \alpha \leq d-1$ let $\Lambda_\alpha$ be an isomorphic copy of $\Lambda_0$ with vertices $w_\alpha,w_{\alpha+d},\ldots ,w_{\alpha+n-d}$, $w_\alpha',w_{\alpha+d}',\ldots ,w_{\alpha+n-d}'$ and edges $(w_i,w_{i+a})$, $(w_i',w_{i+b}')$, $(w_i,w_{i+q-q_0}')$, where $i=dj+\alpha$. Then the disjoint union $\Lambda$ of $\Lambda_0,\ldots ,\Lambda_{d-1}$ has vertices $w_0,\ldots,w_{n-1}$, $w_0',\ldots,w_{n-1}'$ and edges $(w_i,w_{i+a})$, $(w_i',w_{i+b}')$, $(w_i,w_{i+q-q_0}')$ $a\in \mathcal{A}, b\in \mathcal{B}, q\in \mathcal{Q}$, $0\leq i\leq n-1$. Now for each $i$, setting $v_i=w_i$ and $v_i'=w_{i-q_0}'$ shows that $\Lambda$ is isomorphic to $\Gamma_n(\mathcal{A},\mathcal{B},\mathcal{Q})$, and the proof is complete.
\end{proof}

Thus, by Lemmas~\ref{lem:3457isGamma} and \ref{lem:connectedGamma}, in studying graphs that satisfy ($3'$),($4$),($5$),($7$), we may assume that the graphs are connected, as we do in Theorem~\ref{mainthm:planarWhiteheadgraphs} and Corollary~\ref{maincor:allconditions}.

Heuberger~\cite{Heuberger} classified the planar circulant graphs, and this classification will be a key tool in our proofs. Since the inclusion of $0$ in $S$ clearly has no bearing on planarity, we may assume that $0\not \in S$.

\begin{theorem}[{\cite[Theorem~23]{Heuberger}}]\label{thm:planarcirculant}
Let $\mathrm{circ}_n(\mathcal{S})$ be a connected, properly given, circulant graph where $\mathcal{S}\subseteq \{1,\ldots , n-1\}$. Then $\mathrm{circ}_n(\mathcal{S})$ is planar if and only if one of the following holds:
\begin{itemize}
  \item[(i)] $\mathcal{S}=\{s\}$, for some $s$;
  \item[(ii)] $\mathcal{S}=\{s,\pm 2s\}$, for some $s$ and $n$ is even;
  \item[(iii)] $\mathcal{S}=\{s,n/2\}$, where $2|s$ and $n\equiv 2$~mod~$4$.
\end{itemize}
\end{theorem}

Note that for $n\geq 5$, in Theorem~\ref{thm:planarcirculant}(ii) we have $n/2\not \in S$, for otherwise either $s=n/2$, in which case $2s\equiv 0$~mod~$n$, so $S$ is not a subset of $\{1,\ldots ,n-1\}$, or $2s=n/2$ so the fact that $\mathrm{circ}_n(S)$ is connected implies that $(n,\pm n/4,n/2)=1$, i.e.\,$n=4$, a contradiction. Thus for $n\geq 5$ the cases~(ii) and~(iii) of Theorem~\ref{thm:planarcirculant} are distinct. Since the subgraphs $\mathrm{circ}_n(\mathcal{A}),\mathrm{circ}_n(\mathcal{B})$ of our (connected) graph $\Gamma_n(\mathcal{A},\mathcal{B},\mathcal{Q})$ may be disconnected, it is useful to remove the hypothesis that $\mathrm{circ}_n(S)$ is connected from Theorem~\ref{thm:planarcirculant}. Using Lemma~\ref{lem:connectedcirc} to do this we have:

\begin{corollary}
\label{cor:planarcirculant}
Let $\mathcal{S}\subseteq \{1,\ldots , n-1\}$ be properly given. Then $\mathrm{circ}_n(S)$ is planar if and only if one of the following holds:
\begin{itemize}
  \item[(i)] $\mathcal{S}=\emptyset$ or $\mathcal{S}=\{s\}$, for some $s$;
  \item[(ii)] $\mathcal{S}=\{s,\pm 2s\}$, for some $s$ and $n/d$ is even, $d=(n,s)$;
  \item[(iii)] $\mathcal{S}=\{s,n/2\}$, where $2d|s$ and $n/d\equiv 2$~mod~$4$, $d=(n/2,s)$.
\end{itemize}
\end{corollary}

\section{Proof of Theorem~\ref{mainthm:planarWhiteheadgraphs}}\label{sec:mainproof}

We first list certain non-planar graphs $\Gamma_n(\mathcal{A},\mathcal{B},\mathcal{Q})$ that are important for our arguments.

\begin{prop}\label{prop:nonplanargraphs}
The following graphs are non-planar where $n\geq 4$:
\begin{itemize}
  \item[(i)] $\Gamma_n(\emptyset,\emptyset,\{0,1,n-1\})$, where $n$ is odd;
  \item[(ii)] $\Gamma_n(\{1\},\{1\},\{0,1,n-1\})$;
  \item[(iii)] $\Gamma_n(\{1\},\{1\},\{0,2\})$;
  \item[(iv)] $\Gamma_n(\{2\},\{2\},\{0,1\})$, where $n$ is even;
  \item[(v)] $\Gamma_n(\{1,2\},\{1\},\{0\})$, where $n$ is even;
  \item[(vi)] $\Gamma_n(\{2\},\{2\},\{0,n/2\})$, where $n\equiv 2$~mod~$4$;
  \item[(vii)] $\Gamma_n(\{2,n/2\},\{2,n/2\},\{0\})$, where $n\equiv 2$~mod~$4$;
  \item[(viii)] $\Gamma_n(\emptyset,\emptyset,\{0,n/4,n/2,3n/4\})$, where $4|n$;
  \item[(ix)] $\Gamma_n(\emptyset,\emptyset,\{0,n/3,n/2,2n/3\})$, where $6|n$;
  \item[(x)] $\Gamma_n(\emptyset,\emptyset,\{0,n/6,n/2,5n/6\})$, where $6|n$.
\end{itemize}
\end{prop}

\begin{proof}
\begin{itemize}
  \item[(i)] By setting (for each $0\leq i\leq n-1$) $u_i=v_i$ when $i$ is even, $u_i=v_i'$ when $i$ is odd,  $u_{i+n}=v_i'$ when $i$ is even, $u_{i+n}=v_i$ when $i$ is odd, we see that $\Gamma_n(\emptyset,\emptyset,\{0,1,-1\})$ is isomorphic to $\mathrm{circ}_{2n}(\{1,n\})$ which is non-planar (by Corollary~\ref{cor:planarcirculant}).

  \item[(ii)] Contracting the edges $(v_i,v_{i+1}),(v_i',v_{i+1}')$ $(3\leq i\leq n-2$) contracts $\Gamma=\Gamma_n(\{1\},\{1\},\{0,1,-1\})$ to $\Gamma_3(\{1\},\{1\},\{0,1,-1\})$, which contains $\Gamma_3(\emptyset,\emptyset,\{0,1,-1\})$ as a subgraph, but this is isomorphic to the complete bipartite graph $K_{3,3}$.

  \item[(iii)] Contracting the edges $(v_i,v_{i+1}), (v_i',v_{i+1}')$ $(4\leq i\leq n-2)$ contracts $\Gamma=\Gamma_n(\{1\},\{1\},\{0,2\})$ to $\Gamma_4(\{1\},\{1\},\{0,2\})$. Then contracting the edges $(v_i,v_i')$ ($1\leq i \leq 3$) leaves complete graph $K_5$, so $\Gamma$ is non-planar.

  \item[(iv)] Contracting the edges $(v_i,v_{i+2}),(v_i',v_{i+2}')$ $(3\leq i\leq n-3$) contracts $\Gamma=\Gamma_n(\{2\},\{2\},\{0,1\})$ to $\Gamma_4(\{2\},\{2\},\{0,1\})$. Contracting the edges $(v_0,v_0'),(v_2,v_2')$ leaves the complete bipartite graph $K_{3,3}$, so $\Gamma$ is non-planar.
  \item[(v)] Contracting the edges $(v_i,v_{i+1}),(v_i',v_{i+1}')$ $(3\leq i\leq n-2$) contracts $\Gamma=\Gamma_n(\{1,2\},\{1\},\{0\})$
      to $\Gamma_4(\{1,2\},\{1\},\{0\})$. Contracting the edges $(v_2,v_2'),(v_3,v_3'), (v_0',v_1')$ leaves the complete graph $K_5$, so $\Gamma$ is non-planar.

  \item[(vi)] Contracting the edges $(v_i,v_{i+2})$, $(v_i',v_{i+2}')$, for $0\leq i\leq n-1$, $i$ odd and for $4\leq i\leq n-2$, $i$ even, then contracting the edges $(v_i,v_i')$ ($i=0,2,4$) leaves the complete graph $K_5$, so $\Gamma$ is non-planar.

  \item[(vii)] Contracting the edges $(v_i,v_{i+2}), (v_i',v_{i+2}')$ for each odd $i$ with $1\leq i\leq n-1$ and contracting the edges $(v_i,v_{i+2}), (v_i',v_{i+2}')$ for each even $i$ with $4\leq i\leq n-2$ (which is no edges when $n=6$) then contracting the edges $(v_0,v_{0}')$, $(v_2,v_{2}')$, $(v_4,v_{4}')$ leaves the complete graph $K_5$, so $\Gamma$ is non-planar.

  \item[(viii)] The induced subgraph of $\Gamma=\Gamma_n(\emptyset,\emptyset,\{0,n/4,n/2,3n/4\})$ with vertices $v_0,v_{n/4},v_{n/2}, v_0',v_{n/4}',v_{n/2}'$ is the complete bipartite graph $K_{3,3}$ so $\Gamma$ is non-planar.

  \item[(ix)] The induced subgraph of $\Gamma=\Gamma_n(\emptyset,\emptyset,\{0,n/3,n/2,2n/3\})$ with vertices $v_0,v_{n/3},v_{2n/3}, v_0',v_{n/3}',v_{2n/3}'$ is the complete bipartite graph $K_{3,3}$ so $\Gamma$ is non-planar.

  \item[(x)] The map $v_i\mapsto v_i$, $v_i'\mapsto v_{i+n/2}'$ ($0\leq i\leq n-1$) gives an isomorphism between this graph and that in part~(ix), which is non-planar.
\end{itemize}
\end{proof}

It simplifies our later arguments if we may assume $n\geq 5$ so we now classify the connected planar graphs $\Gamma_n(\mathcal{A},\mathcal{B},\mathcal{Q})$ in the cases $n=2,3,4$.

\begin{lemma}\label{lem:n=2or3}
Suppose $\mathcal{Q}\neq \emptyset$, $0\not \in \mathcal{A},\mathcal{B}$ and assume that $\Gamma=\Gamma_n(\mathcal{A},\mathcal{B},\mathcal{Q})$ is properly given. Then
\begin{itemize}
  \item[1.] $\Gamma_2(\mathcal{A},\mathcal{B},\mathcal{Q})$ is connected and planar if and only if $(\{\mathcal{A},\mathcal{B}\},\mathcal{Q})$ is any of the cases from Class~I of Table~\ref{table:planarWHG}.

  \item[2.] $\Gamma_3(\mathcal{A},\mathcal{B},\mathcal{Q})$ is connected and planar if and only if $(\{\mathcal{A},\mathcal{B}\},\mathcal{Q})$ is any of the cases from Class~I of Table~\ref{table:planarWHG}.

  \item[3.] $\Gamma_4(\mathcal{A},\mathcal{B},\mathcal{Q})$ is connected and planar if and only if $(\{\mathcal{A},\mathcal{B}\},\mathcal{Q})$ is any of the cases from Class~I or Class~II of Table~\ref{table:planarWHG}.

\end{itemize}
\end{lemma}

\begin{proof}
\noindent 1. If $\mathcal{A}=\mathcal{B}=\emptyset$ and $|\mathcal{Q}|=1$ then $\Gamma=\Gamma_2(\mathcal{A},\mathcal{B},\mathcal{Q})$ is disconnected; in all other cases $\Gamma$ is connected and planar (of type~(I.1),(I.2),(I.3),(I.4), or~(I.5)).

\noindent 2. Here $\mathcal{A},\mathcal{B}\subseteq \{\pm 1\}$. If $|\mathcal{Q}|=1$ then $\Gamma$ is connected and planar if and only if $\mathcal{A},\mathcal{B}$ are not both empty; that is, in cases (I.3) and (I.4). If $|\mathcal{Q}|=2$ then $\Gamma$ is connected and planar in all cases; that is, in cases~(I.1),(I.2) and (I.5). If $|\mathcal{Q}|=3$ then $\Gamma$ contains $\Gamma_3(\emptyset,\emptyset,\{0,1,2\})$, which is the complete graph $K_{3,3}$, so it is non-planar.

\noindent 3. We have $\mathcal{A}\subseteq\{\pm 1,2\}$, $\mathcal{B}\subseteq\{\pm 1,2\}$ and there exists $q\in \mathcal{Q}$ for some $q \in \{0,1,2,3\}$.
  If $|\mathcal{Q}|=4$ then $\Gamma$ contains the (non-planar) complete bipartite graph $K_{4,4}$ so we may assume $|\mathcal{Q}|\leq 3$.

  \noindent \textbf{Case~1: $|\mathcal{A}|+|\mathcal{B}|\geq 3$.} If $\{ \mathcal{A},\mathcal{B}\}=\{ \{\pm 1, 2\}, \{\pm 1,2\} \}$ or $\{ \{\pm 1, 2\}, \{\pm 1\} \}$ then $\Gamma$ contains a subgraph isomorphic to $\Gamma_4( \{1, 2\}, \{1\}, \{0\} )$, which is non-planar by Proposition~\ref{prop:nonplanargraphs}(v). Thus we have $\{ \mathcal{A},\mathcal{B}\}=\{ \{\pm 1, 2\}, \{2\} \}$. If $q\pm 1 \in \mathcal{Q}$ then $\Gamma$ contains a subgraph isomorphic to $\Gamma_4( \{2\},\{2\},\{0,1\})$ which is non-planar by Proposition~\ref{prop:nonplanargraphs}(iv). Thus $\mathcal{Q}=\{q\}$ or $\{q,q+2\}$, in which case $\Gamma$ is connected and planar (of type~(II.12) or~(II.14)).

  \noindent \textbf{Case~2: $|\mathcal{A}|+|\mathcal{B}|=2$.} Suppose first that $|\mathcal{A}|=|\mathcal{B}|=1$ so either $\{\mathcal{A},\mathcal{B}\}=\{ \{\pm 1\}, \{\pm 1\}\}$ or $\{\{\pm 1\}, \{2\}\}$ or $\{\{2\},\{2\}\}$. Suppose $\mathcal{A}=\{\pm 1\}$, $\mathcal{B}=\{\pm 1\}$. If $q,q+2 \in \mathcal{Q}$ then $\Gamma$ contains $\Gamma_4(\{1\},\{1\},\{0,2\})$, which is non-planar by Proposition~\ref{prop:nonplanargraphs}(iii). If $\mathcal{Q}=\{q,q+1,q-1\}$ then $\Gamma$ contains $\Gamma_4(\emptyset,\emptyset,\{q,q+1,q-1\})$ which is non-planar by Proposition~\ref{prop:nonplanargraphs}(ii). If $\mathcal{Q}=\{q\}$ or $\{q,q\pm 1\}$ then $\Gamma$ is planar (of type~(I.3) or~(I.1)). Suppose that $\{\mathcal{A},\mathcal{B}\}=\{ \{\pm 1\}, \{2\}\}$. Then $\Gamma$ is planar for any $\mathcal{Q}$ with $|\mathcal{Q}|\leq 3$ (of type~(II.1),(II.7),(II.11), or~(II.15)). Suppose that $\{\mathcal{A},\mathcal{B}\}=\{ \{2\}, \{2\}\}$. Then if $\mathcal{Q}=\{q\}$ or $\{q,q+2\}$ then $\Gamma$ is disconnected; if $\mathcal{Q}=\{q,q\pm 1\}$ or $\{q,q+1,q-1\}$ then $\Gamma$ contains a subgraph isomorphic to $\Gamma_4(\{2\},\{2\},\{0,1\})$, which is non-planar by Proposition~\ref{prop:nonplanargraphs}(iv). Suppose then that $\mathcal{A}=\emptyset$ or $\mathcal{B}=\emptyset$. Then $\{\mathcal{A},\mathcal{B}\}=\{\emptyset,\{\pm 1,2\}\}$ and $\Gamma$ is connected and properly given, of type~(II.2),(II.5),(II.10) or (II.13).

  \noindent \textbf{Case~3: $|\mathcal{A}|+|\mathcal{B}|<2$.} Suppose first that $\mathcal{A}=\mathcal{B}=\emptyset$. If $\mathcal{Q}=\{q\}$ or $\{q,q+2\}$ then $\Gamma$ is disconnected. If $\mathcal{Q}=\{q,q\pm 1\}$ or $|\mathcal{Q}|=3$ then $\Gamma$ is connected and planar (of type~(I.5) or~(II.6)). Suppose then that $|\mathcal{A}|+|\mathcal{B}|=1$; then without loss of generality we may assume $\mathcal{A}=\emptyset$. If $\mathcal{B}=\{2\}$ and either $\mathcal{Q}=\{q\}$ or $\{q,q+2\}$ then $\Gamma$ is disconnected; if $\mathcal{B}=\{2\}$ and $\mathcal{Q}=\{q,q\pm 1\}$ then $\Gamma$ is connected and planar (of type~(II.8)); if $\mathcal{B}=\{2\}$ and $\mathcal{Q}=\{q,q+1,q-1\}$ then $\Gamma$ is planar (of type~(II.4)). Thus we may assume that $\mathcal{B}=\{\pm 1\}$, in which case $\Gamma$ is connected and planar of type~(I.2),(I.4),(II.3) or~(II.9).
\end{proof}

In the next three lemmas we give further conditions for non-planarity of $\Gamma$.

\begin{lemma}\label{lem:nonplanarsizesPART0}
Suppose $0\not \in \mathcal{A},\mathcal{B}$ and assume that $\Gamma=\Gamma_n(\mathcal{A},\mathcal{B},\mathcal{Q})$ is properly given. If $|\mathcal{A}|>2$ or $|\mathcal{B}|>2$ then $\Gamma$ is non-planar.
\end{lemma}

\begin{proof}
The properly given circulant graphs $\mathrm{circ}_n(\mathcal{A})$, $\mathrm{circ}_n(\mathcal{B})$ are subgraphs of $\Gamma_n(\mathcal{A},\mathcal{B},\mathcal{Q})$ and are non-planar by Corollary~\ref{cor:planarcirculant}.
\end{proof}

\begin{lemma}\label{lem:nonplanarsizesPART1}
Suppose $0\not \in \mathcal{A},\mathcal{B}$, $n\geq 5$, and assume that $\Gamma=\Gamma_n(\mathcal{A},\mathcal{B},\mathcal{Q})$ is properly given. If $|\mathcal{Q}|>3$ then $\Gamma$ is non-planar.
\end{lemma}

\begin{proof}
Let $q_0\in \mathcal{Q}$. For each $0\leq i\leq n-1$ set $w_i=v_i$, $w_i'=v_{i+q_0}'$ and set $\mathcal{R}=\{q-q_0\ |\ q\in \mathcal{Q}\}$, so $0\in \mathcal{R}$. Then $\Gamma$ has vertices $w_i,w_i'$ ($0\leq i\leq n-1$) and edges $(w_i,w_{i+a})$, $(w_i',w_{i+b}')$, $(w_i,w_{i+r}')$ where $a\in \mathcal{A}$, $b\in \mathcal{B}$, $r\in \mathcal{R}$, so $\Gamma$ is isomorphic to $\Gamma_n(\mathcal{A},\mathcal{B},\mathcal{R})$ and hence we may assume that $0\in \mathcal{Q}$.

Suppose that $|\mathcal{Q}|=4$, and let $\mathcal{Q}=\{0,q,r,s\}$, say, where $0,q,r,s$ are distinct mod~$n$ and suppose that $\Gamma$ is connected and planar. Consider the subgraph $\Lambda=\Gamma_n(\emptyset,\emptyset,\mathcal{Q})=\Gamma_n(\emptyset,\emptyset,\{0,q,r,s\})$ of $\Gamma$.

Let $C^{(0)}$ be the graph obtained from $\Lambda$ by contracting the edges $(v_i,v_i')$;
let $C^{(q)}$ be the graph obtained from $\Lambda$ by contracting the edges $(v_i,v_{i+q}')$;
let $C^{(r)}$ be the graph obtained from $\Lambda$ by contracting the edges $(v_i,v_{i+r}')$;
let $C^{(s)}$ be the graph obtained from $\Lambda$ by contracting the edges $(v_i,v_{i+s}')$ ($0\leq i \leq n-1$).
Then $C^{(0)}$, $C^{(q)}$, $C^{(r)}$, $C^{(s)}$, are the (not necessarily properly given) circulant graphs $\mathrm{circ}_n(S^{(0)})$, $\mathrm{circ}_n(S^{(q)})$, $\mathrm{circ}_n(S^{(r)})$, $\mathrm{circ}_n(S^{(s)})$, where $S^{(0)}=\{q,r,s\}$, $S^{(q)}=\{-q,r-q,s-q\}$,
$S^{(r)}=\{-r,q-r,s-r\}$, $S^{(s)}=\{-s,q-s,r-s\}$.

Since $0,q,r,s$ are distinct mod~$n$ we have $|S^{(0)}|=|S^{(q)}|=|S^{(r)}|=|S^{(s)}|=3$ so if any of $C^{(0)}$,  $C^{(q)}$,  $C^{(r)}$,  $C^{(s)}$ are properly given then it is non-planar (by Corollary~\ref{cor:planarcirculant}) and so $\Gamma$ is non-planar. The conditions for these circulants to be properly given is as follows (all congruences are mod~$n$):\\

\begin{tabular}{ll}
$C^{(0)}$ : & $q \not \equiv -r$ and $q\not \equiv -s$ and $r\not \equiv -s$;\\
$C^{(q)}$ : & $r \not \equiv 2q$ and $s\not \equiv 2q$ and $r+s\not \equiv 2q$;\\
$C^{(r)}$ : & $q \not \equiv 2r$ and $s\not \equiv 2r$ and $q+s\not \equiv 2r$;\\
$C^{(s)}$ : & $q \not \equiv 2s$ and $r\not \equiv 2s$ and $q+r\not \equiv 2s$.\\
\end{tabular}

\medskip

Suppose $n/2 \in \mathcal{Q}$. Without loss of generality we may set $q=n/2$ and therefore $r,s\not \equiv n/2$~mod~$n$. Then the conditions for $C^{(0)}$,  $C^{(q)}$ to be properly given each become $r\not \equiv -s$~mod~$n$, so we may assume $s\equiv -r$~mod~$n$ in which case the conditions for $C^{(r)}$,  $C^{(s)}$ to be properly given become $n/2 \not \equiv 2r$ and $3r\not \equiv 0,n/2$. Thus we may assume $r\equiv \pm n/4$ or $r\equiv \pm n/3$~mod~$n$ or $r\equiv \pm n/6$~mod~$n$. That is, ($q=n/2$, $r=\pm n/4$, $s=\mp n/4$) or ($q=n/2$, $r=\pm n/3$, $s=\mp n/3$) or ($q=n/2$, $r=\pm n/6$, $s=\mp n/6$) so $\Lambda=\Gamma_n(\emptyset,\emptyset, \{0,n/2, n/4,-n/4\})$ or $\Lambda=\Gamma_n(\emptyset,\emptyset,\{0,n/2,n/3,-n/3\})$ or $\Lambda=\Gamma_n(\emptyset,\emptyset,\{0,n/2,n/6,-n/6\})$, in which case $\Lambda$ is non-planar by Proposition~\ref{prop:nonplanargraphs}(viii),(ix),(x).

Suppose then that $n/2 \not \in \mathcal{Q}$. If $q \not \equiv -r$ and $q\not \equiv -s$ and $r\not \equiv -s$ then $C^{(0)}$ is properly given, and hence non-planar, so without loss of generality we may assume $s=n-r$. Noting that $n-q\not \equiv q$, $n-r\not \equiv r$, $n-s\not \equiv s$, (because $n/2 \not \in \mathcal{Q}$) the conditions for  $C^{(q)}$,  $C^{(r)}$,  $C^{(s)}$ to be properly given become\\

\begin{tabular}{ll}
$C^{(q)}$ : & $r \not \equiv 2q$ and $-r\not \equiv 2q$;\\
$C^{(r)}$ : & $q \not \equiv +2r$ and $3r\not \equiv 0$ and $q\not \equiv 3r$;\\
$C^{(s)}$ : & $q \not \equiv -2r$ and $3r\not \equiv 0$ and $q\not \equiv -3r$.\\
\end{tabular}

\medskip

We have $C^{(0)}=\mathrm{circ}_n(\{q,r\})$ and $q\not \equiv \pm r$~mod~$n$ (since $q,r,s$ are distinct mod~$n$) so Corollary~\ref{cor:planarcirculant} gives that $C^{(0)}$ is planar only if (i) $q\equiv \pm 2r$ and $n/(n,r)$ is even or (ii) $r\equiv \pm 2q$ and $n/(n,q)$ is even.

\noindent \textbf{Case 1. $q\equiv \pm 2r$~mod~$n$.} Then the conditions for $C^{(q)}$ to be properly given become $3r \not \equiv 0$ and $5r\not \equiv 0$. If $3r\equiv 0$~mod~$n$ then $q\equiv \mp r$~mod~$n$, so $s\equiv \pm q$~mod~$n$ and we have a contradiction as $q,r,s$ are distinct mod~$n$. If $5r\equiv 0$~mod~$n$ then $(n,r)=n/5$ so $n/(n,r)=5$ is odd so $C^{(0)}$ is non-planar.

\noindent \textbf{Case 2. $r\equiv 2q$~mod~$n$ (resp. $r\equiv -2q$~mod~$n$).} Then the conditions for $C^{(r)}$ (resp. $C^{(s)}$) to be properly given become
\( 3q\not \equiv 0, 6q\not \equiv 0, 5q \not \equiv 0.\)
If $3q\equiv 0$ then $r\equiv \pm q$ so $s\equiv \mp q$ and we have a contradiction as $q,r,s$ are distinct mod~$n$. If $5q\equiv 0$~mod~$n$ then $(n,q)=n/5$ so $n/(n,q)=5$ is odd so $C^{(0)}$ is non-planar. If $6q\equiv 0$~mod~$n$ then $2q\equiv \pm n/3$ and $\Lambda$ contains $\Gamma_n(\emptyset,\emptyset,\{0,2q,-2q\})$ whose connected components are isomorphic to $\Gamma_{3}(\emptyset,\emptyset,\{0,1,-1\})$, which is the complete bipartite graph $K_{3,3}$ so $\Lambda$ is non-planar.
\end{proof}

\begin{lemma}\label{lem:nonplanarsizesPART2}
Suppose $\mathcal{Q}\neq \emptyset$, $0\not \in \mathcal{A},\mathcal{B}$ and assume that $\Gamma=\Gamma_n(\mathcal{A},\mathcal{B},\mathcal{Q})$ is properly given. If  $|\mathcal{A}|=|\mathcal{B}|=2$ then $\Gamma$ is non-planar.
\end{lemma}

\begin{proof}
Let $\mathcal{A}=\{a_1,a_2\}$, $\mathcal{B}=\{b_1,b_2\}$. Since $\Gamma$ is properly given we have $a_1\not \equiv \pm a_2$, $b_1\not \equiv \pm b_2$. Suppose $|\mathcal{Q}|=1$, that is, $\mathcal{Q}=\{q\}$. Suppose for contradiction that $\Gamma$ is connected and planar. Contracting the edges $(v_i,v_{i+q}')$ contracts $\Gamma$ to the connected, planar, circulant graph $\Lambda=\mathrm{circ}_n(\mathcal{A}\cup \mathcal{B})$. If for any $j\in \{1,2\}$ we have $b_j\not \in \{\pm a_1, \pm a_2\}$ then the subgraph $\mathrm{circ}_n(\mathcal{A}\cup \{b_j\})$ of $\Lambda$ is properly given and non-planar by Corollary~\ref{cor:planarcirculant}. Therefore we may assume $\mathcal{B}=\{\pm a_1,\pm a_2\}$ and so $\Lambda=\mathrm{circ}_n(\mathcal{A})$, which is connected, planar, and properly given.

Suppose $n/2\not \in \mathcal{A}$. Then by Corollary~\ref{cor:planarcirculant} we have, without loss of generality, that $a_2\equiv \pm 2a_1$~mod~$n$, and since $\Lambda$ is connected we have $(n,a_1)=1$ and $n$ is even. Therefore $\Gamma$ is isomorphic to \linebreak $\Gamma_n(\{1,2\},\{1,2\},\{0\})$ which contains $\Gamma_n(\{1,2\},\{1\},\{0\})$, which is non-planar by Proposition~\ref{prop:nonplanargraphs}(v).
Suppose then that $n/2 \in \mathcal{A}$, $a_2=n/2$, say. Then $\mathcal{A}=\{a_1,n/2\}$ and since $\Lambda=\mathrm{circ}_n(\mathcal{A})$ is connected and planar, Theorem~\ref{thm:planarcirculant} gives that $(n/2,a_1)=1$, $n\equiv 2$~mod~$4$, and $a_1$ is even. Then $\Gamma$ is isomorphic to $\Gamma_n(\{2,n/2\},\{2,n/2\},\{q\})$ which is non-planar by Proposition~\ref{prop:nonplanargraphs}(vii). If $|\mathcal{Q}|>1$ then $\Gamma$ contains $\Gamma_n(\mathcal{A},\mathcal{B},\{q\})$ which (as we have just shown) is non-planar.
\end{proof}

Consider now the case $\mathcal{A}=\mathcal{B}=\emptyset$.

\begin{lemma}\label{lem:A=B=empty}
Let $\mathcal{A}=\mathcal{B}=\emptyset$, $n\geq 5$, $\mathcal{Q}\neq \emptyset$. Then $\Gamma=\Gamma_n(\mathcal{A},\mathcal{B},\mathcal{Q})$ is connected and planar if and only if either
\begin{itemize}
  \item[(II.5)] $\mathcal{Q}=\{q,q+s\}$, for some $q$, where $(n,s)=1$; or
  \item[(II.6)] $\mathcal{Q}=\{q,q+s,q-s\}$, for some $q$, where $n$ is even and $(n,s)=1$.
\end{itemize}
\end{lemma}

\begin{proof}
If $|\mathcal{Q}|=1$ then $\Gamma$ is disconnected. If $|\mathcal{Q}|>3$ then $\Gamma$ is non-planar by Lemma~\ref{lem:nonplanarsizesPART1}.

Suppose that $|\mathcal{Q}|=2$, i.e. $\mathcal{Q}=\{q,q+s\}$ for some $q,s \in \{0,\ldots ,n-1\}$, $s\not \equiv 0$~mod~$n$.  For each $i\in \{0,\ldots ,n-1\}$ let $w_i=v_i$, $w_i'=v_{i+q}'$. Then $\Gamma$ has vertices $w_i,w_i'$ and edges $(w_i,w_i')$, $(w_i,w_{i+s}')$,  ($0\leq i\leq n-1$). Then $\Gamma$ consists of $(n,s)$ cycles $w_i'-w_i-w_{i+s}'-w_{i+s}-w_{i+2s}'-w_{i+2s}-\cdots - w_{i-s}'-w_{i-s}-w_i'$, so it is connected and planar (of type~(I.5)) if and only if $(n,s)=1$.

Now suppose that $|\mathcal{Q}|=3$ and let $\mathcal{Q}=\{p,p+s,p+r\}$ for some $p,r,s\in \{0,\ldots,n-1\}$, where $s\not \equiv 0$, $r\not \equiv 0$, $r\not \equiv s$~mod~$n$.

Suppose first that $s\equiv -r$ or $r\equiv 2s$ or $s\equiv 2r$~mod~$n$. If $s\equiv-r$ let $q=p,t=s$; if $r\equiv 2s$~mod~$n$ let $q=p+s,t=s$; if $s\equiv 2r$ let $q=p+r$, $t=r$. Then $\mathcal{Q}=\{q,q+t,q-t\}$. We claim that if $\Gamma$ is connected then it is isomorphic to $\Gamma_n\{\emptyset,\emptyset,\{0,1,-1\})$, so if $n$ is odd $\Gamma$ is non-planar by Proposition~\ref{prop:nonplanargraphs}(i), and if $n$ is even then $\Gamma$ is planar (of type~(II.6)). Now $\Gamma$ has vertices $v_i,v_i'$ and edges $(v_i,v_{i+q}')$, $(v_i,v_{i+q+t}')$, $(v_i,v_{i+q-t}')$. For each $0\leq i\leq n-1$ let $w_i=v_i$, $w_i'=v_{i+q}'$; then $\Gamma$ has vertices $w_i,w_i'$ and edges $(w_i,w_{i}')$, $(w_i,w_{i+t}')$, $(w_i,w_{i-t}')$. Therefore $\Gamma=\Gamma_n(\emptyset,\emptyset,\{0,t,-t\})$, which (by Lemma~\ref{lem:connectedGamma}) is connected if and only if $(n,t)=1$, in which case $\Gamma$ is isomorphic to $\Gamma_n\{\emptyset,\emptyset,\{0,1,-1\})$, as claimed.

Suppose then that $s\not \equiv -r$, $r\not \equiv 2s$, and $s\not \equiv 2r$~mod~$n$, and suppose that $\Gamma$ is connected and planar. For each $0\leq i\leq n-1$ let $w_i=v_i$, $w_i'=v_{i+p}'$. Then $\Gamma$ has edges $(w_i,w_i')$, $(w_i,w_{i+s}')$, $(w_i,w_{i+r}')$ so $\Gamma$ is isomorphic to $\Gamma_n(\emptyset,\emptyset,\{0,s,r\})$. Contracting the edges $(w_i,w_i')$ of $\Gamma$ and removing loops leaves the circulant $C^{(0)}=\mathrm{circ}_n(\{s,r\})$, which is properly given since $s\not \equiv \pm r$~mod~$n$. Contracting the edges $(w_i,w_{i+s}')$ of $\Gamma$ leaves the circulant $C^{(s)}=\mathrm{circ}_n(-s,r-s)$, which is properly given since $r\not \equiv 0$, $r\not \equiv 2s$~mod~$n$. Contracting the edges $(w_i,w_{i+r}')$ of $\Gamma$ leaves the circulant $C^{(r)}=\mathrm{circ}_n(-r,s-r)$, which is properly given since $s\not \equiv 0$, $s\not \equiv 2r$~mod~$n$.

If $s=n/2$ then $C^{(0)}=\mathrm{circ}_n(\{n/2,r\})$, which is connected and planar, so Theorem~\ref{thm:planarcirculant} gives that $(n/2,r)=1$, $n\equiv 2$~mod~$4$, and $r$ is even, so in particular $r-n/2$ is odd. In the same way $C^{(s)}=\mathrm{circ}_n(\{n/2,r-n/2\})$ is connected and planar so $(n/2,r-n/2)=1$, $n\equiv 2$~mod~$4$, and $r-n/2$ is even, a contradiction. If $r=n/2$ then $C^{(0)}=\mathrm{circ}_n(\{n/2,s\})$, which is connected and planar, so Theorem~\ref{thm:planarcirculant} gives that $(n/2,s)=1$, $n\equiv 2$~mod~$4$, and $s$ is even, so in particular $s-n/2$ is odd. In the same way $C^{(r)}=\mathrm{circ}_n(\{n/2,s-n/2\})$ is connected and planar so $(n/2,s-n/2)=1$, $n\equiv 2$~mod~$4$, and $s-n/2$ is even, a contradiction.
Thus we may assume that $n/2 \not \in \{r,s\}$. Then since $C^{(0)}$ is connected and planar, noting that $r\not \equiv 2s$, $s\not \equiv 2r$~mod~$n$, Theorem~\ref{thm:planarcirculant} gives that either $s\equiv -2r$ or $r\equiv -2s$~mod~$n$ and so $C^{(s)}=\mathrm{circ}_n(\{2r,3r\})$ and $C^{(r)}=\mathrm{circ}_n(\{2s,3s\})$, respectively, which are not connected and planar since $n\geq 5$, a contradiction.
\end{proof}

\begin{lemma}\label{lem:A=emptyset}
Let $\mathcal{A}=\emptyset$ or $\mathcal{B}=\emptyset$, $\mathcal{A}\cup \mathcal{B}\neq \emptyset$, $|\mathcal{Q}|=1$ or $2$, $n\geq 5$, $0\not \in \mathcal{A}$, $0\not \in \mathcal{B}$ and suppose that $\Gamma=\Gamma_n(\mathcal{A},\mathcal{B}, \mathcal{Q})$ is properly given. Then $\Gamma$ is connected and planar if and only if one of the following holds:
\begin{itemize}
  \item[(I.2)] $\{\mathcal{A},\mathcal{B}\}=\{\emptyset, \{\pm s\}\}$, $\mathcal{Q}=\{q,q+s\}$, for some $q$, $(n,s)=1$;
  \item[(I.4)] $\{\mathcal{A},\mathcal{B}\}=\{\emptyset,\{s\}\}$, $\mathcal{Q}=\{q\}$, for some $q$ where $(n,s)=1$;
  \item[(II.5)] $\{\mathcal{A},\mathcal{B}\}=\{\emptyset, \{s,\pm 2s\}\}$, $\mathcal{Q}=\{q\}$, for some $q$ where $n$ is even, $(n,s)=1$;
  \item[(II.8)] $\{\mathcal{A},\mathcal{B}\}=\{\emptyset, \{\pm 2s\}\}$, $\mathcal{Q}=\{q,q+s\}$, for some $q$ where $n$ is even, $(n,s)=1$;
  \item[(II.9)] $\{\mathcal{A},\mathcal{B}\}=\{\emptyset, \{\pm s\}\}$, $\mathcal{Q}=\{q,q+2s\}$, for some $q$ where $n$ is even, $(n,s)=1$;
  \item[(II.10)] $\{\mathcal{A},\mathcal{B}\}=\{\emptyset, \{\pm s, \pm 2s\}\}$, $\mathcal{Q}=\{q,q+s\}$, for some $q$, $n$ even, $(n,s)=1$;
  \item[(II.13)] $\{\mathcal{A},\mathcal{B}\}=\{\emptyset, \{\pm s, \pm 2s\}\}$, $\mathcal{Q}=\{q,q+2s\}$, for some $q$, $n$ even, $(n,s)=1$;
  \item[(III.2)] $\{\mathcal{A},\mathcal{B}\}=\{\emptyset, \{n/2, s\}\}$, $\mathcal{Q}=\{q,q+n/2\}$, for some $q$, where $n,s$ are even, $(n/2,s)=1$;
  \item[(III.3)] $\{\mathcal{A},\mathcal{B}\}=\{\emptyset, \{s\}\}$, $\mathcal{Q}=\{q,q+n/2\}$, for some $q$, where $n,s$ are even, $(n/2,s)=1$;
  \item[(III.9)] $\{\mathcal{A},\mathcal{B}\}=\{\emptyset, \{s,n/2\}\}$, $\mathcal{Q}=\{q\}$, for some $q$ where $n,s$ are even $(n/2,s)=1$;
  \item[(III.12)] $\{\mathcal{A},\mathcal{B}\}=\{\emptyset, \{n/2,\pm s\}\}$, $\mathcal{Q}=\{q,q+s\}$, for some $q$, where $n,s$ are even, $(n/2,s)=1$;
  \item[(III.14)] $\{\mathcal{A},\mathcal{B}\}=\{\emptyset, \{n/2\}\}$, $\mathcal{Q}=\{q,q+s\}$, for some $q$, where $n,s$ are even, $(n/2,s)=1$.
  \end{itemize}
\end{lemma}

\begin{proof}
By setting $\mathcal{Q}'=\{n-q\ |\ q\in \mathcal{Q}\}$ we see that the graphs $\Gamma_n(\mathcal{A},\mathcal{B},\mathcal{Q})$, $\Gamma_n(\mathcal{B},\mathcal{A},\mathcal{Q}')$ are isomorphic, so without loss of generality we may assume that $\mathcal{A}=\emptyset$.

If $|\mathcal{Q}|=1$ then $\Gamma$ is connected and planar if and only if $\mathrm{circ}_n(\mathcal{B})$ is connected and planar, the conditions for which are given in Theorem~\ref{thm:planarcirculant}, and which correspond to cases (I.4),(II.5),(III.9).

If $|\mathcal{Q}|=2$, $\mathcal{Q}=\{q,q+t\}$, say, then each vertex $v_i$ is of degree~2. Replacing each length 2 path $v_i'-v_{i-q}-v_{i+t}'$ by an edge $(v_i',v_{i+t}')$ we see that $\Gamma$ is connected and planar if and only if $C=\mathrm{circ}_n(\mathcal{B} \cup \{t\})$ is connected and planar, so in particular $|\mathcal{B}|\leq 2$.
Suppose $|\mathcal{B}|=1$, $\mathcal{B}=\{b\}$, say, then $C=\mathrm{circ}_n(\{b,t\})$. If $b\equiv \pm t$ then $C$ is connected and planar if and only if $(n,t)=1$, giving case~(I.2). If $b\not \equiv \pm t, n/2$~mod~$n$ then $C$ is connected and planar if and only if $n$ is even and either $b\equiv \pm 2t$ and $(n,t)=1$ or $t\equiv \pm 2b$ and $(n,b)=1$, giving cases~(II.8),(II.9), respectively. If $b\equiv n/2$ then $C$ is connected and planar if and only if $t$ is even and $(n/2,t)=1$, giving case~(III.14). If $t\equiv n/2$ then $C$ is connected and planar if and only if $b$ is even and $(n/2,b)=1$, giving case (III.3).
Suppose $|\mathcal{B}|=2$; then since $C$ is planar we have $\pm t\in \mathcal{B}$, so $C=\mathrm{circ}_n(\{b,t\})$ and $\mathcal{B}=\{b,\pm t\}$, $b\not \equiv \pm t$ (since $\Gamma$ is properly given). If $b,t\not \equiv n/2$~mod~$n$ then either $b\equiv \pm 2t$ and $(n,t)=1$ or $t\equiv \pm 2b$ and $(n,b)=1$, giving cases~(II.10),(II.13), respectively. If $b\equiv n/2$~mod~$n$ then $t$ is even, $(n/2,t)=1$, giving case~(III.12); if $t\equiv n/2$~mod~$n$ then $b$ is even, $(n/2,b)=1$, giving case~(III.2).
\end{proof}

We now consider the case $|\mathcal{Q}|=3$.

\begin{lemma}\label{lem:|Q|=3}
Let $|\mathcal{Q}|=3$, $n\geq 5$, $\mathcal{A}\cup \mathcal{B} \neq \emptyset$, $0\not \in \mathcal{A},\mathcal{B}$ and suppose that $\Gamma=\Gamma_n(\mathcal{A},\mathcal{B},\mathcal{Q})$ is properly given. Then $\Gamma$ is connected and planar if and only if the following holds:
\begin{itemize}
 \item[(II.1)] $\{\mathcal{A},\mathcal{B}\} =\{ \{\pm s\}, \{\pm 2s\} \}$, $\mathcal{Q}=\{q,q+s,q-s\}$, for some $q$, and $n$ is even, $(n,s)=1$;
 \item[(II.2)] $\{\mathcal{A},\mathcal{B}\}=\{\emptyset, \{\pm s, \pm 2s\}\}$, $\mathcal{Q}=\{q,q+s,q-s\}$, for some $q$, and $n$ is even, $(n,s)=1$;
 \item[(II.3)] $\{\mathcal{A},\mathcal{B}\}=\{\emptyset,\{\pm s\}\}$, $\mathcal{Q}=\{q,q+s,q-s\}$, for some $q$, and $n$ is even, $(n,s)=1$;
 \item[(II.4)] $\{\mathcal{A},\mathcal{B}\}=\{\emptyset, \{\pm 2s\}\}$, $\mathcal{Q}=\{q,q+s,q-s\}$, for some $q$, and $n$ is even, $(n,s)=1$.
 \end{itemize}
\end{lemma}

\begin{proof}
Suppose that $\Gamma=\Gamma_n(\mathcal{A},\mathcal{B},\mathcal{Q})$ is connected and planar and let $\mathcal{Q}=\{p,p+s,p+r\}$ where $p,q,r\in\{0,\ldots , n-1\}$, $s\not \equiv 0$, $r\not \equiv 0$, $r\not \equiv s$~mod~$n$. If $s\not \equiv -r$, $r\not \equiv 2s$, and $s\not \equiv 2r$~mod~$n$ then $\mathcal{Q}\neq \{q,q+t,q-t\}$ for any $q,t$ so the subgraph $\Gamma_n(\emptyset,\emptyset,\mathcal{Q})$ is non-planar by Lemma~\ref{lem:A=B=empty}.

Thus we may assume that $s\equiv -r$ or $r\equiv 2s$ or $s\equiv 2r$~mod~$n$ and set $q=p,t=s$ (when $s\equiv -r$), $q=p+s, t=s$ (when $r\equiv 2s$), and $q=p+r, t=r$ (when $s\equiv 2r$). Then $\mathcal{Q}=\{q,q+t,q-t\}$. Contracting the edges $(v_i,v_{i+q-t}')$ leaves the connected, planar, circulant graph $\Lambda=\mathrm{circ}_n(S)$ where $S=\mathcal{A}\cup \mathcal{B}\cup \{t,2t\}$. By Theorem~\ref{thm:planarcirculant} we have $\mathcal{A},\mathcal{B}\subseteq \{\pm t,\pm 2t\}$ and then $\Lambda$ has $(n,t)$ connected components so $(n,t)=1$ and $n$ is even.

If $\pm t\in \mathcal{A}\cap \mathcal{B}$ then $\Gamma$ contains the subgraph $\Gamma_n(\{t\},\{t\},\{q,q+t,q-t\})$ which is isomorphic to $\Gamma_n(\{1\},\{1\},\{0,1,n-1\})$ which is non-planar by Proposition~\ref{prop:nonplanargraphs}(ii).
If $\pm 2t\in \mathcal{A}\cap \mathcal{B}$ then $\Gamma$ contains the subgraph $\Gamma_n(\{2t\},\{2t\},\{q,q+t,q-t\})$ which is isomorphic to $\Gamma_n(\{2\},\{2\},\{0,1,-1\})$, which contains the subgraph $\Gamma_n(\{2\},\{2\},\{0,1\})$ which is non-planar by Proposition~\ref{prop:nonplanargraphs}(iv). Thus if $\mathcal{A}\neq \emptyset$ and $\mathcal{B}\neq \emptyset$ then $\{\mathcal{A},\mathcal{B}\}= \{\{\pm t \},\{\pm 2t \}\}$ so $\Gamma$ is isomorphic to $\G_n\{ \{1\}, \{2\}, \{q,q+1,q-1\}\}$ which is planar (of type~(II.1)) since~$n$ is even, and is connected by Lemma~\ref{lem:connectedGamma}. Therefore we may assume that $\{\mathcal{A},\mathcal{B}\}=\{\emptyset, \mathcal{X}\}$ where $\mathcal{X}\subseteq \{\pm t, \pm 2t\}$.  If $\mathcal{X}=\{\pm t,\pm 2t\}, \{\pm t\}, \{\pm 2t\}$ then we get cases (II.2),(II.3),(II.4), respectively.
\end{proof}

By Lemmas~\ref{lem:n=2or3}--\ref{lem:|Q|=3} the connected and planar graphs $\Gamma_n(\mathcal{A},\mathcal{B},\mathcal{Q})$ have been classified except when all of the following hold: $n\geq 5$ and $1\leq |\mathcal{Q}|\leq 2, 2\leq |\mathcal{A}|+|\mathcal{B}| \leq 3$, $\mathcal{A}\neq \emptyset$, $\mathcal{B}\neq \emptyset$. We deal with these cases in the following four lemmas. For the case $|\mathcal{A}|+|\mathcal{B}| = 3$, $|\mathcal{Q}|=2$; that is, for the case $\{|\mathcal{A}|,|\mathcal{B}|\}=\{2,1\}$, $|\mathcal{Q}|=2$ we have the following.

\begin{lemma}\label{lem:|A|=2|B|=1|Q|=2}
Let $\{|\mathcal{A}|,|\mathcal{B}|\}=\{2,1\}$, $0\not \in \mathcal{A},\mathcal{B}$, $|\mathcal{Q}|=2$, $q\in \mathcal{Q}$, $n\geq 5$ and suppose that $\Gamma_n(\mathcal{A},\mathcal{B},\mathcal{Q})$ is properly given. Then $\Gamma$ is connected and planar if and only if one of the following holds:
\begin{itemize}
  \item[(II.12)] $\{\mathcal{A},\mathcal{B}\}=\{ \{\pm s,\pm 2s\}, \{\pm 2s\}\}$, $\mathcal{Q}=\{q,q+2s\}$, $n$ is even, $(n,s)=1$;
  \item[(III.1)] $\{\mathcal{A},\mathcal{B}\}=\{ \{ s,n/2\}, \{n/2\}\}$, $\mathcal{Q}=\{q,q+n/2\}$, $(n/2,s)=1$, $s$ is even;
  \item[(III.7)] $\{\mathcal{A},\mathcal{B}\}=\{ \{\pm s,n/2\}, \{\pm s\}\}$, $\mathcal{Q}=\{q,q+s\}$, $(n/2,s)=1$, $s$ is even;
  \item[(III.11)] $\{\mathcal{A},\mathcal{B}\}=\{ \{ \pm s,n/2\}, \{n/2\}\}$, $\mathcal{Q}=\{q,q+s\}$, $(n/2,s)=1$, $s$ is even.
\end{itemize}
\end{lemma}

\begin{proof}
Without loss of generality we may assume $|\mathcal{A}|=2$, $|\mathcal{B}|=1$, so let $\mathcal{A}=\{a_1,a_2\}$ ($a_1\not \equiv \pm a_2$~mod~$n$), $\mathcal{B}=\{b\}$, $\mathcal{Q}=\{q,q+s\}$ ($s\not \equiv 0$~mod~$n$) and suppose that $\Gamma$ is connected and planar. If $s\not \equiv \pm a_1,\pm a_2$ (mod~$n$) or $b\not \equiv \pm a_1,\pm a_2$ then contracting the edges $(v_i,v_{i+q}')$ leaves the non-planar circulant $\mathrm{circ}_n(\mathcal{A}\cup \mathcal{B} \cup \{s\})$, so we may assume (without loss of generality) that $s\equiv \pm a_2$~mod~$n$ and either $b=a_1$ or $b=a_2$.

Suppose $n/2 \not \in \mathcal{A}$. Then since the subgraph $\mathrm{circ}_n(\mathcal{A})$ of $\Gamma_n(\mathcal{A},\mathcal{B},\mathcal{Q})$ is planar (but not necessarily connected), Corollary~\ref{cor:planarcirculant} implies that $n$ is even and either $a_2\equiv \pm 2a_1$ or $a_1\equiv \pm 2a_2$. Thus we have the following four cases:
\begin{itemize}
  \item[1.] $\mathcal{A}=\{a_1,\pm 2a_1\}$, $\mathcal{B}=\{a_1\}$, $s=\pm 2a_1$;
  \item[2.] $\mathcal{A}=\{a_2,\pm 2a_2\}$, $\mathcal{B}=\{\pm 2a_2\}$, $s=\pm a_2$;
  \item[3.] $\mathcal{A}=\{a_1,\pm 2a_1\}$, $\mathcal{B}=\{\pm 2a_1\}$, $s=\pm 2a_1$;
  \item[4.] $\mathcal{A}=\{a_2,\pm 2a_2\}$, $\mathcal{B}=\{a_2\}$, $s=\pm a_2$.
\end{itemize}
Since $\Gamma$ is connected, by Lemma~\ref{lem:connectedGamma} we have $(n,a_1)=1$ in Cases~1,3 and $(n,a_2)=1$ in Cases~2,4. Therefore, in Cases 1 and~4, the graph $\Gamma_n(\mathcal{A},\mathcal{B},\mathcal{Q})$ is isomorphic to $\Gamma_n(\{1,2\},\{1\},\{0,2\})$, $\Gamma_n(\{1,2\},\{1\},\{0,1\})$, respectively, each of which which contains the subgraph $\Gamma_n(\{1,2\},\{1\},\{0\})$, which is non-planar by Proposition~\ref{prop:nonplanargraphs}(v). In Case~2 the graph $\Gamma_n(\mathcal{A},\mathcal{B},\mathcal{Q})$ is isomorphic to $\Gamma_n(\{2,1\},\{2\},\{0,1\})$, which contains $\Gamma_n(\{2\},\{2\},\{0,1\})$ which is non-planar by Proposition~\ref{prop:nonplanargraphs}(iv). In Case~3 the graph $\Gamma_n(\mathcal{A},\mathcal{B},\mathcal{Q})$ has a planar embedding (of type (II.12)).

Suppose then that $n/2 \in \mathcal{A}$. Then $\mathcal{A}=\{a,n/2\}$ and so $\mathcal{B}=\{a\}$ or $\mathcal{B}=\{n/2\}$ and we have $\mathcal{Q}=\{q,q\pm a\}$ or $\mathcal{Q}=\{q,q+n/2\}$. Then since $\Gamma$ is connected we have $(n/2,a)=1$. Then the subgraph $\mathrm{circ}_n(\mathcal{A})=\mathrm{circ}_n(a,n/2)$ of $\Gamma_n(\mathcal{A},\mathcal{B},\mathcal{Q})$ is connected and planar so Theorem~\ref{thm:planarcirculant} implies that $n\equiv 2$~mod~$4$ and $a$ is even. Thus we have the following four cases:
\begin{itemize}
  \item[5.] $\mathcal{A}=\{a,n/2\}$, $\mathcal{B}=\{a\}$, $s=\pm a$;
  \item[6.] $\mathcal{A}=\{a,n/2\}$, $\mathcal{B}=\{n/2\}$, $s=n/2$;
  \item[7.] $\mathcal{A}=\{a,n/2\}$, $\mathcal{B}=\{a\}$, $s=n/2$;
  \item[8.] $\mathcal{A}=\{a,n/2\}$, $\mathcal{B}=\{n/2\}$, $s=\pm a$.
\end{itemize}
In Case~7 the graph $\Gamma$ is isomorphic to $\Gamma_n\{ \{2,n/2\}, \{2\}, \{0,n/2\}\}$ which contains $\Gamma_n\{ \{2\}, \{2\}, \{0,n/2\}\}$ which is non-planar by Proposition~\ref{prop:nonplanargraphs}(vi). In the remaining cases $\Gamma$ is planar of type~(III.1),(III.7) or (III.11).
\end{proof}

\begin{lemma}\label{lem:|A|=1|B|=1|Q|=2}
Let $|\mathcal{A}|=1$, $|\mathcal{B}|=1$, $0\not \in \mathcal{A},\mathcal{B}$, $|\mathcal{Q}|=2$, $q\in \mathcal{Q}$, $n\geq 5$ and suppose that $\Gamma=\Gamma_n(\mathcal{A},\mathcal{B},\mathcal{Q})$ is properly given. Then $\Gamma$ is connected and planar if and only if one of the following holds:
\begin{itemize}
  \item[(I.1)] $\{\mathcal{A},\mathcal{B}\} = \{\{\pm s\},\{\pm s\}\}$, $\mathcal{Q}=\{q,q+s\}$, $(n,s)=1$;
  \item[(II.7)] $\{\mathcal{A},\mathcal{B}\} = \{ \{\pm 2s\},\{\pm s\} \}$, $\mathcal{Q}=\{q,q+s\}$, $n$ is even, $(n,s)=1$;
  \item[(II.15)] $\{\mathcal{A},\mathcal{B}\} = \{ \{\pm 2s\},\{\pm s\} \}$, $\mathcal{Q}=\{q,q+2s\}$, $n$ is even, $(n,s)=1$;
  \item[(III.5)] $\{\mathcal{A},\mathcal{B}\} = \{ \{n/2\},\{s\} \}$, $\mathcal{Q}=\{q,q+n/2\}$, $n,s$ are even, $(n/2,s)=1$;
  \item[(III.10)] $\{\mathcal{A},\mathcal{B}\} = \{ \{n/2\},\{\pm s\} \}$, $\mathcal{Q}=\{q,q+s\}$, $n,s$ are even, $(n/2,s)=1$;
  \item[(III.13)] $\{ \mathcal{A},\mathcal{B}\} = \{\{n/2\},\{n/2\}\}$, $\mathcal{Q}=\{q,q+s\}$, $n,s$ are even, $(n/2,s)=1$.
\end{itemize}
\end{lemma}

\begin{proof}
Let $\mathcal{A}=\{a\}, \mathcal{B}=\{b\}, \mathcal{Q}=\{q,q+t\}$ where $t \not \equiv 0$~mod~$n$, and suppose that $\Gamma=\Gamma_n(\mathcal{A},\mathcal{B},\mathcal{Q})$ is connected and planar.
Contracting the edges $(v_i,v_{i+q}')$ leaves the connected and planar (but not necessarily properly given) circulant graph $\Lambda=\mathrm{circ}_n(S)$ where $S=\{a,b,t\}$.

\noindent \textbf{Case 1: $t\equiv n/2, a\not \equiv \pm b$~mod~$n$.} Here $\Lambda=\mathrm{circ}_n(S)$ where $S=\{a,b,n/2\}$. Theorem~\ref{thm:planarcirculant} implies that either $a\equiv n/2$ or $b\equiv n/2$~mod~$n$. Thus $\{\mathcal{A},\mathcal{B}\} = \{\{n/2\},\{s\}\}$ for some $s\not \equiv n/2$~mod~$n$. Since $n\geq 5$ Theorem~\ref{thm:planarcirculant} implies that $s$ is even, $n\equiv 2$~mod~$4$, $(n/2,s)=1$, in which case $\Gamma$ is connected and planar (of type (III.5)).

\noindent \textbf{Case 2: $t\equiv n/2, a\equiv \pm b$~mod~$n$.} Here $\Lambda=\mathrm{circ}_n(S)$ where $S=\{s,n/2\}$, where $s=a$. Since $\Lambda$ is connected we have $(n/2,s)=1$ (which implies, in particular, that $s\not \equiv n/2$~mod~$n$). Theorem~\ref{thm:planarcirculant} implies that $n\equiv 2$~mod~$4$, $s$ is even, $(n/2,s)=1$, in which case $\Gamma$ is isomorphic to $\Gamma_n(\{2\},\{2\},\{0,n/2\})$, which is non-planar by Proposition~\ref{prop:nonplanargraphs}(vi).

\noindent \textbf{Case 3: $t\not \equiv n/2, a\not \equiv \pm b$, $a\not \equiv n/2, b\not \equiv n/2$~mod~$n$.} Here $\Lambda=\mathrm{circ}_n(S)$ where $S=\{a,b,t\}$ so Theorem~\ref{thm:planarcirculant} implies that $a\equiv \pm t$ or $b\equiv \pm t$~mod~$n$ so $\Lambda=\mathrm{circ}_n\{s,t\}$ where $s \not \equiv \pm t$, $s\not \equiv n/2$~mod~$n$, $\{\mathcal{A},\mathcal{B}\} = \{ \{\pm s\}, \{\pm t\} \}$. Since $\Lambda$ is connected and planar we have $n$ is even and either $s\equiv \pm 2t$ or $t \equiv \pm 2s$~mod~$n$. In the case $s\equiv \pm 2t$~mod~$n$ we have $(n,t)=1$ (since $\Lambda$ is connected), in which case $\Gamma$ is connected and planar (of type~(II.7)). In the case $t\equiv \pm 2s$~mod~$n$ we have $(n,s)=1$ (since $\Lambda$ is connected), in which case $\Gamma$ is connected and planar (of type~(II.15)).

\noindent \textbf{Case 4: $t\not \equiv n/2, a\not \equiv \pm b$~mod~$n$, and ($a\equiv n/2$ or $b\equiv n/2$).} Here $\{\mathcal{A},\mathcal{B}\} = \{ \{n/2\}, \{s\}\}$ (for some $s\neq n/2$) so $\Lambda=\mathrm{circ}_n(S)$ where $S=\{n/2,s,t\}$ so Theorem~\ref{thm:planarcirculant} implies that $s\equiv \pm t$~mod~$n$ so $\Lambda=\mathrm{circ}_n(\{n/2,s\})$. Since $\Lambda$ is connected and planar we have $(n/2,s)=1$, $s$ is even, and $n\equiv 2$~mod~$4$. In this case $\Gamma$ is connected and planar (of type~(III.10)).

\noindent \textbf{Case 5: $t\not \equiv n/2, a\equiv \pm b$, $t\not \equiv \pm a$~mod~$n$.} Here $\mathcal{A}= \{\pm s\}$, $\mathcal{B}= \{\pm s\}$  (for some $s$) so $\Lambda=\mathrm{circ}_n(S)$ where $S=\{s,t\}$. Suppose $s\equiv n/2$~mod~$n$. Then Theorem~\ref{thm:planarcirculant} implies that $t$ is even, $n\equiv 2$~mod~$4$, $(n/2,t)=1$, in which case $\Gamma$ is planar (of type~(III.13)). Suppose then that $s\not \equiv n/2$~mod~$n$. Then Theorem~\ref{thm:planarcirculant} implies that $n$ is even, $s\equiv \pm 2t$ or $t\equiv \pm 2s$~mod~$n$. In the case $s\equiv \pm 2t$~mod~$n$, we have $(n,t)=1$ (since $\Lambda$ is connected) so we may assume $t=1$, in which case $\Gamma$ is isomorphic to $\Gamma_n(\{2\},\{2\},\{0,1\})$, which is non-planar by Proposition~\ref{prop:nonplanargraphs}(iv). In the case $t\equiv \pm 2s$~mod~$n$, we have $(n,s)=1$ (since $\Lambda$ is connected) so we may assume $s=1$ so $\Gamma$ is isomorphic to $\Gamma_n(\{1\},\{1\},\{0,2\})$ which is non-planar by Proposition~\ref{prop:nonplanargraphs}(iii).

\noindent \textbf{Case 6: $t\not \equiv n/2, a\equiv \pm b$, $t\equiv \pm a$~mod~$n$.} Here $\mathcal{A}=\{\pm t\}$, $\mathcal{B}=\{\pm t\}$, $\mathcal{Q}=\{q,q+t\}$ so since $\Gamma$ is connected we have $(n,t)=1$, in which case $\Gamma$ is connected and planar (of type~(I.1)).
\end{proof}

\begin{lemma}\label{lem:|A|=2|B|=1|Q|=1}
Let $\{|\mathcal{A}|, |\mathcal{B}|\}=\{2,1\}$, $0\not \in \mathcal{A},\mathcal{B}$, $\mathcal{Q}=\{q\}$, $n\geq 5$ and suppose that $\Gamma=\Gamma_n(\mathcal{A},\mathcal{B},\mathcal{Q})$ is properly given. Then $\Gamma$ is connected and planar if and only if one of the following holds:
\begin{itemize}
  \item[(II.14)] $\{\mathcal{A},\mathcal{B}\}=\{\{s,\pm 2s\}, \{\pm 2s\}\}$,  where $n$ is even, $(n,s)=1$;
  \item[(III.4)] $\{\mathcal{A},\mathcal{B}\}=\{ \{s,n/2\}, \{n/2\}\}$ where $s$ is even, $(n/2,s)=1$;
  \item[(III.8)] $\{\mathcal{A},\mathcal{B}\}=\{ \{s,n/2\}, \{\pm s\}\}$ where $s$ is even, $(n/2,s)=1$.
\end{itemize}
\end{lemma}

\begin{proof}
Without loss of generality we may assume $|\mathcal{A}|=2$, $|\mathcal{B}|=1$, so let $\mathcal{A}=\{a_1,a_2\}$ ($a_1\not \equiv \pm a_2$~mod~$n$), $\mathcal{B}=\{b\}$, $\mathcal{Q}=\{q\}$ and suppose that $\Gamma_n(\mathcal{A},\mathcal{B},\mathcal{Q})$ is connected and planar. Contracting the edges $(v_i,v_{i+q}')$ leaves the connected, planar circulant graph $\Lambda=\mathrm{circ}_n(\{a_1,a_2,b\})$. Since $\Lambda$ is planar Theorem~\ref{thm:planarcirculant} implies that $b\equiv \pm a_1$ or $b\equiv \pm a_2$~mod~$n$. Without loss of generality let $b=a_1$, so $\Lambda=\mathrm{circ}_n(\{a_1,a_2\})=\mathrm{circ}_n(\mathcal{A})$ and $\Gamma=\Gamma_n(\{a_1,a_2\},\{a_1\},\{q\})$. Since $\Lambda$ is connected we have $(n,a_1,a_2)=1$ and since it is planar Theorem~\ref{thm:planarcirculant} implies that one of the following holds:
\begin{itemize}
  \item[1.] $\mathcal{A}=\{a,\pm 2a\}$, $\mathcal{B}=\{\pm 2a\}$, where $n$ is even, $(n,a)=1$;
  \item[2.] $\mathcal{A}=\{a,\pm 2a\}$, $\mathcal{B}=\{\pm a\}$, where $n$ is even, $(n,a)=1$;
  \item[3.] $\mathcal{A}=\{a,n/2\}$, $\mathcal{B}=\{n/2\}$, where $n\equiv 2$~mod~$4$, $a$ is even, $(n/2,a)=1$;
  \item[4.] $\mathcal{A}=\{a,n/2\}$, $\mathcal{B}=\{\pm a\}$, where $n\equiv 2$~mod~$4$, $a$ is even, $(n/2,a)=1$.
\end{itemize}
In case~2 $\Gamma$ is isomorphic to $\Gamma_n(\{1,2\},\{1\},\{0\})$, which is non-planar by Proposition~\ref{prop:nonplanargraphs}(v). In cases~1,3,4 the graph $\Gamma$ is connected and planar of type~(II.14),(III.4),(III.8), respectively.
\end{proof}

\begin{lemma}\label{lem:|A|=|B|=|Q|=1}
Let $|\mathcal{A}|=|\mathcal{B}|=1$, $0\not \in \mathcal{A},\mathcal{B}$, $\mathcal{Q}=\{q\}$, $n\geq 5$ and suppose that $\Gamma=\Gamma_n(\mathcal{A},\mathcal{B},\mathcal{Q})$ is properly given. Then $\Gamma$ is connected and planar if and only if one of the following holds:
\begin{itemize}
  \item[(I.3)] $\{\mathcal{A},\mathcal{B}\}=\{\{s\},\{\pm s\}\}$, $(n,s)=1$;
  \item[(II.11)] $\{\mathcal{A},\mathcal{B}\}=\{\{s\},\{\pm 2s\}\}$, $n$ is even, $(n,s)=1$;
  \item[(III.6)] $\{\mathcal{A},\mathcal{B}\}=\{\{s\},\{n/2\}\}$ where $s$ is even, $(n/2,s)=1$.
\end{itemize}
\end{lemma}

\begin{proof}
Let $\mathcal{A}=\{a\},\mathcal{B}=\{b\},\mathcal{Q}=\{q\}$ and suppose that $\Gamma_n(\mathcal{A},\mathcal{B},\mathcal{Q})$  is connected and planar. If $a \equiv \pm b$~mod~$n$ then $\Gamma$ has $(n,a)$ components, so is connected if and only if $(n,a)=1$, in which case there is a planar embedding of $\Gamma$ (of type~(I.3)). Suppose then that $a \not \equiv \pm b$~mod~$n$. Contracting the edges $(v_i,v_{i+q}')$  ($0\leq i \leq n-1$) contracts $\Gamma$ to the connected, planar, properly given, circulant $\mathrm{circ}_n(\mathcal{A}\cup \mathcal{B})=\mathrm{circ}_n(\{a,b\})$ which, by Theorem~\ref{thm:planarcirculant}, is connected and planar if and only if either
\begin{itemize}
  \item[1.] $\{\mathcal{A},\mathcal{B}\}=\{\{s\},\{\pm 2s\}\}$, where $(n,s)=1$, $n$ is even; or
  \item[2.] $\{\mathcal{A},\mathcal{B}\}=\{\{s\},\{n/2\}\}$, where $s$ is even, $(n/2,s)=1$, $n\equiv 2$~mod~$4$.
\end{itemize}
In each of these cases $\Gamma$ is connected and planar (of type~(II.11) or~(III.6), respectively).
\end{proof}

We are now in a position to prove Theorem~\ref{mainthm:planarWhiteheadgraphs}.

\begin{proof}[Proof of Theorem~\ref{mainthm:planarWhiteheadgraphs}.]
Figures~\ref{fig:(I.1)}--\ref{fig:(III.14)} show that $\Gamma_n(\mathcal{A},\mathcal{B},\mathcal{Q})$ is planar in all cases of Table~\ref{table:planarWHG}, and by Lemma~\ref{lem:connectedGamma} it is connected. Suppose then that $\Gamma$ is connected and planar.

If $n\leq 4$ then the result follows from Lemma~\ref{lem:n=2or3} so we may assume that $n\geq 5$. Lemmas~\ref{lem:nonplanarsizesPART0},\ref{lem:nonplanarsizesPART1},\ref{lem:nonplanarsizesPART2} give that $|\mathcal{A}|\leq 2$, $|\mathcal{B}|\leq 2$, $|\mathcal{Q}|\leq 3$, $\{|\mathcal{A}|,|\mathcal{B}|\}\neq \{2,2\}$. By Lemma~\ref{lem:A=B=empty} we may assume that $\mathcal{A}$ and $\mathcal{B}$ are not both empty. Then if $|\mathcal{Q}|\leq 2$ and either $\mathcal{A}=\emptyset$ or $\mathcal{B}=\emptyset$ then the result follows from Lemma~\ref{lem:A=emptyset}. If $|\mathcal{Q}|=3$ then the result follows from Lemma~\ref{lem:|Q|=3}. Therefore we may assume $\mathcal{A}\neq \emptyset$, $\mathcal{B}\neq \emptyset$, $|\mathcal{Q}|\leq 2$, so in particular $|\mathcal{A}|+|\mathcal{B}|\geq 2$.
If $|\mathcal{A}|+|\mathcal{B}|=3$ and $|\mathcal{Q}|=2$ then the result follows from Lemma~\ref{lem:|A|=2|B|=1|Q|=2};
if $|\mathcal{A}|+|\mathcal{B}|=2$ and $|\mathcal{Q}|=2$ then the result follows from Lemma~\ref{lem:|A|=1|B|=1|Q|=2};
if $|\mathcal{A}|+|\mathcal{B}|=3$ and $|\mathcal{Q}|=1$ then the result follows from Lemma~\ref{lem:|A|=2|B|=1|Q|=1};
and if $|\mathcal{A}|+|\mathcal{B}|=2$ and $|\mathcal{Q}|=1$ then the result follows from Lemma~\ref{lem:|A|=|B|=|Q|=1}.
\end{proof}

Corollary~\ref{maincor:allconditions} follows immediately from Theorem~\ref{mainthm:planarWhiteheadgraphs} and Lemma~\ref{lem:3457isGamma}.

\section*{Acknowledgement}

We thank the referee for the helpful comments that have improved the paper.

\bigskip

  \textsc{Department of Mathematics and Maxwell Institute for Mathematical Sciences, Heriot-Watt University, Edinburgh
EH14 4AS, UK.}\par\nopagebreak
  \textit{E-mail address}, \texttt{J.Howie@hw.ac.uk}

  \medskip

  \textsc{Department of Mathematical Sciences, University of Essex, Wivenhoe Park, Colchester, Essex CO4 3SQ, UK.}\par\nopagebreak
  \textit{E-mail address}, \texttt{Gerald.Williams@essex.ac.uk}

\clearpage

\begin{table}
\begin{itemize}
\item \textbf{Class~I:} $(n,s)=1$.
\begin{center}
\begin{tabular}{|L{1.5cm}|L{2.1cm} L{2.0cm}|L{3cm}|}
\hline
Type & $\{ \mathcal{A},$           & $\mathcal{B}\}$           & $\mathcal{Q}$\\\hline
(I.1)  & $\{\{\pm s\}$,         & $\{\pm s\} \}$   & $\{q,q+s\}$\\
(I.2) &  $\{\emptyset ,$ & $ \{\pm s\}\}$ & $\{q,q+s\}$\\
(I.3)  & $\{\{\pm s\}$,         & $\{\pm s\} \}$   & $\{q\}$\\
(I.4) &  $\{\emptyset ,$ & $\{s\}\}$ & $\{q\}$\\
(I.5)  & $\{\emptyset $,    & $\emptyset \}$ & $\{q,q+s\}$\\
\hline
\end{tabular}
\end{center}

\item \textbf{Class~II:} $n$ even, $(n,s)=1$.
\begin{center}
\begin{tabular}{|L{1.5cm}|L{2.1cm} L{2.0cm}|L{3cm}|}
\hline
Type & $\{ \mathcal{A},$           & $\mathcal{B}\}$           & $\mathcal{Q}$\\\hline
(II.1)   & $\{\{s\}$,      & $\{\pm 2s\} \}$    & $\{q,q+s,q-s\}$\\
(II.2) &  $\{\emptyset ,$ & $ \{\pm s, \pm 2s\}\}$ & $\{q,q+s,q-s\}$\\
(II.3) &  $\{\emptyset ,$ & $\{\pm s\}\}$ & $\{q,q+s,q-s\}$\\
(II.4) &  $\{\emptyset ,$ & $ \{\pm 2s\}\}$ & $\{q,q+s,q-s\}$\\
(II.5) &  $\{\emptyset ,$ & $ \{s,\pm 2s\}\}$ & $\{q\}$\\
(II.6)   & $\{\emptyset $, & $\emptyset \}$ & $\{q,q+s,q-s\}$\\
(II.7)   & $\{\{\pm s\}$,      & $\{\pm 2s\} \}$    & $\{q,q+s\}$\\
(II.8) &  $\{\emptyset ,$ & $ \{\pm 2s\}\}$ & $\{q,q+s\}$\\
(II.9) &  $\{\emptyset ,$ & $ \{\pm s\}\}$ & $\{q,q+2s\}$\\
(II.10) &  $\{\emptyset ,$ & $ \{\pm s, \pm 2s\}\}$ & $\{q,q+s\}$\\
(II.11)   & $\{\{s\}$,      & $\{\pm 2s\} \}$    & $\{q\}$\\
(II.12)  & $\{\{\pm s,\pm 2s\}$,       & $\{\pm 2s\} \}$    & $\{q,q+2s\}$\\
(II.13) &  $\{\{\pm s, \pm 2s\}$, & $\emptyset \}$ & $\{q,q+2s\}$\\
(II.14)  & $\{\{s,\pm 2s\}$,       & $\{\pm 2s\} \}$    & $\{q\}$\\
(II.15)  & $\{\{\pm s\}$,          & $\{\pm 2s\} \}$    & $\{q,q+2s\}$\\
\hline
\end{tabular}
\end{center}

\item \textbf{Class~III:} $n\equiv 2$~mod~$4$, $(n/2,s)=1$, $s$ even.

\begin{center}
\begin{tabular}{|L{1.5cm}|L{2.1cm} L{2.0cm}|L{3cm}|}
\hline
Type & $\{ \mathcal{A},$         & $\mathcal{B}\}$         & $\mathcal{Q}$\\\hline
(III.1)  & $\{\{s,n/2\}$,  & $\{n/2\} \}$ & $\{q,q+n/2\}$\\
(III.2) &  $\{\{n/2,s\},$ & $ \emptyset\}$ & $\{q,q+n/2\}$\\
(III.3) &  $\{\{s\},$ & $ \emptyset\}$ & $\{q,q+n/2\}$\\
(III.4)  & $\{\{s,n/2\}$,   & $\{n/2\} \}$ & $\{q\}$\\
(III.5)  & $\{\{s\}$,       & $\{n/2\} \}$ & $\{q,q+n/2\}$\\
(III.6) & $\{\{s\}$,       & $\{n/2\} \}$ & $\{q\}$\\
(III.7)  & $\{\{\pm s,n/2\}$,     & $\{\pm s\} \}$    & $\{q,q+s\}$\\
(III.8)  & $\{\{s,n/2\}$,     & $\{\pm s\} \}$    & $\{q\}$\\
(III.9) &  $\{\{s,n/2\},$ & $ \emptyset\}$ & $\{q\}$\\
(III.10)  & $\{\{n/2\}$,       & $\{\pm s\} \}$    & $\{q,q+s\}$\\
(III.11)  & $\{\{\pm s,n/2\}$,     & $\{n/2\} \}$    & $\{q,q+s\}$\\
(III.12) &  $\{\{\pm s,n/2\},$ & $\emptyset  \}$ & $\{q,q+s\}$\\
(III.13)  & $\{\{n/2\}$,       & $\{n/2\} \}$    & $\{q,q+s\}$\\
(III.14) &  $\{\{n/2\}$ & $ \emptyset\}$ & $\{q,q+s\}$\\
\hline
\end{tabular}
\end{center}

\end{itemize}
\caption{Classification of properly given connected and planar graphs $\Gamma_n(\mathcal{A},\mathcal{B},\mathcal{Q})$. \label{table:planarWHG}}
\end{table}

\clearpage


\begin{figure}
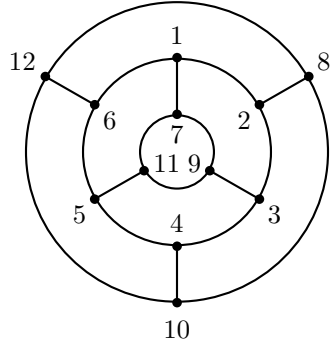

%

\caption{A graph satisfying (1)--(7) but not (8).\label{fig:counterexample}}
\end{figure}

\psset{xunit=1.3cm,yunit=0.5cm,algebraic=true,dimen=middle,dotstyle=o,dotsize=5pt 0,linewidth=0.8pt,arrowsize=3pt 2,arrowinset=0.25}


\begin{figure}
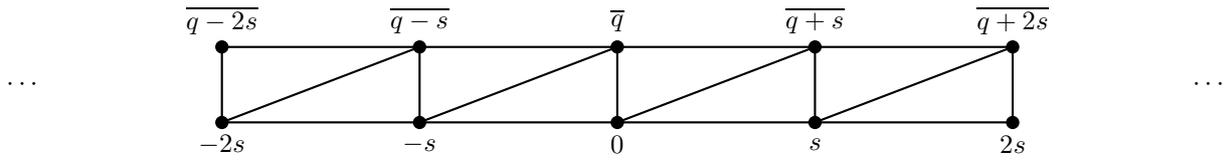

%
\caption{Type~(I.1)\label{fig:(I.1)}}
\end{figure}

\begin{figure}
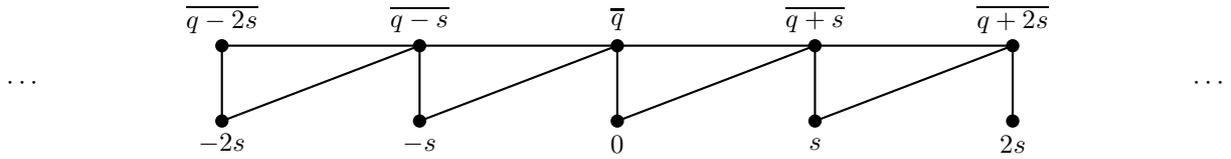

%
\caption{Type~(I.2)\label{fig:(I.2)}}
\end{figure}

\begin{figure}
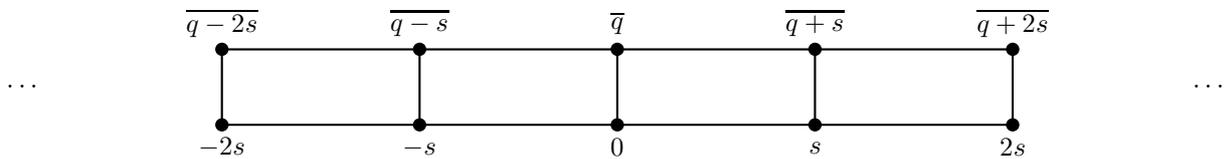

%
\caption{Type~(I.3)\label{fig:(I.3)}}
\end{figure}

\begin{figure}
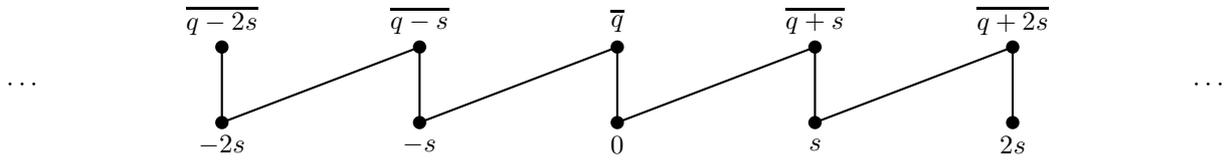

%
\caption{Type~(I.5)\label{fig:(I.5)}}
\end{figure}


\begin{figure}
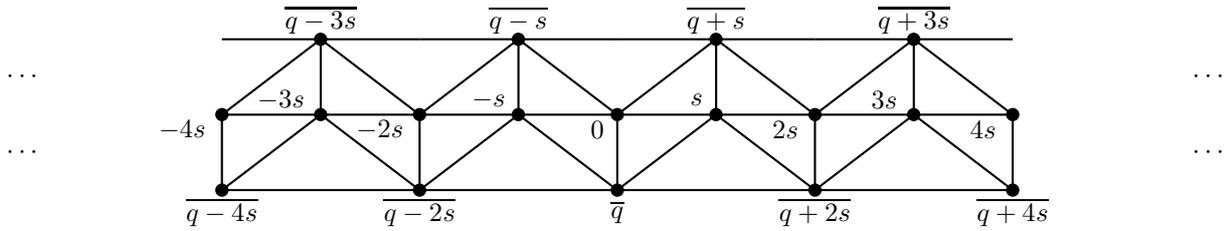

%
\caption{Type~(II.1)\label{fig:(II.1)}}
\end{figure}

\begin{figure}
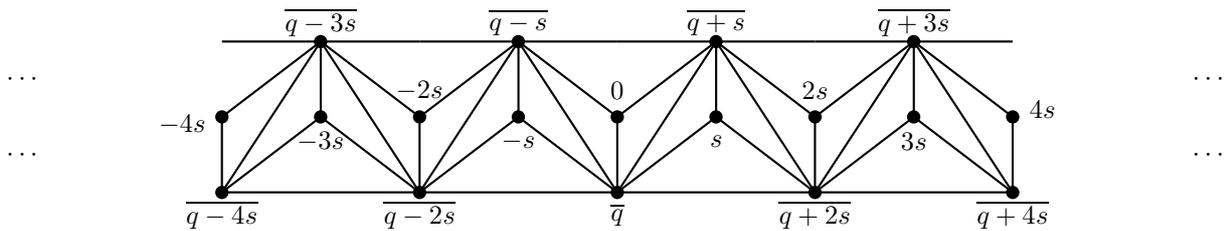

%
\caption{Type~(II.2)\label{fig:(II.2)}}
\end{figure}

\begin{figure}
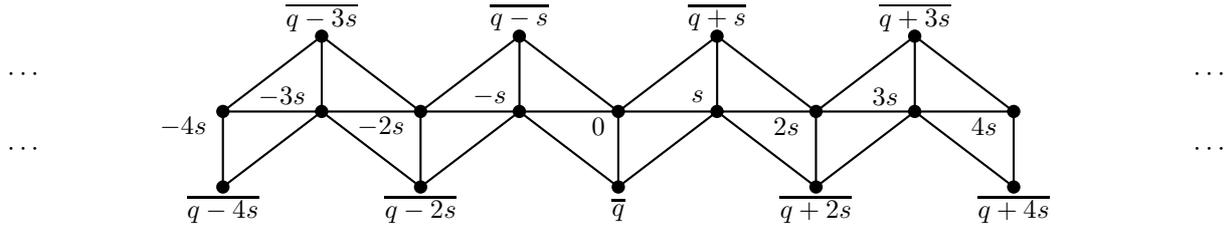

%
\caption{Type~(II.3)\label{fig:(II.3)}}
\end{figure}

\begin{figure}
%
\caption{Type~(II.15)}
\end{figure}

\begin{figure}
%
\caption{Type~(III.1)\label{fig:(III.1)}}
\end{figure}

\begin{figure}
%
\caption{Type~(III.2)\label{fig:(III.2)}}
\end{figure}

\begin{figure}
%
\caption{Type~(III.3)\label{fig:(III.3)}}
\end{figure}

\begin{figure}
%
\caption{Type~(III.4)\label{fig:(III.4)}}
\end{figure}

\begin{figure}
\caption{Type~(III.6)\label{fig:(III.6)}}
\end{figure}

\clearpage

\begin{figure}
%
\caption{Type~(III.14)\label{fig:(III.14)}}
\end{figure}
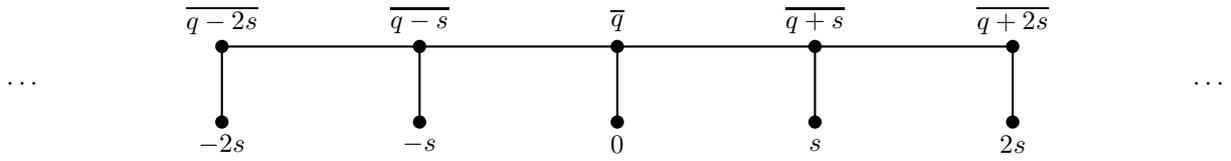
\end{document}